\documentclass[12pt]{amsart}
\usepackage{amsmath,amssymb,amsthm}
\usepackage{verbatim}
\usepackage{libertine}
\usepackage{enumitem}
\usepackage[libertine]{newtxmath}
\usepackage[colorlinks, urlcolor=blue,  citecolor=blue]{hyperref}
\usepackage[all]{xy}%
\usepackage[utf8]{inputenc}
\title[Integral points on affine quadrics]{
	Arithmetic purity of the Hardy-Littlewood property\\and geometric sieve for affine quadrics}
\author{Yang Cao \and Zhizhong Huang}
\date{November 19, 2020}

\address{Yang CAO
	\newline University of Science and Technology of China
	\newline School of Mathematical Sciences
	\newline 96 Jinzhai Road, 230026 Hefei, Anhui, China}
\email{yangcao1988@gmail.com}

\address{Zhizhong HUANG
	\newline Institute of Science and Technology Austria
	\newline Am Campus 1, 3400 Klosterneuburg, Austria}
\email{zhizhong.huang@yahoo.com}

\date{October 28, 2021.}

\def\setminus{\mathchoice
	{\mathbin{\vrule height .92ex width 1.81ex depth -.58ex}}
	{\mathbin{\vrule height .92ex width 1.81ex depth -.58ex}}
	{\mathbin{\vrule height .65ex width 1.00ex depth -.43ex}}
	{\mathbin{\vrule height .50ex width 0.770ex depth -.34ex}}
}

\newcommand{\APHL}{\textbf{(APHL)}}
\newcommand{\APSHL}{\textbf{(APSHL)}}
\newcommand{\APSA}{\textbf{(APSA)}}

\newcommand{\BFQ}{\mathbf{Q}}
\newcommand{\adele}{\mathbf{A}}
\newcommand{\bsquare}{\mathfrak{b}}
\newcommand{\bsquarestar}{\mathfrak{b}^*}

\newcommand{\xx}{\mathbf{x}}
\newcommand{\yy}{\mathbf{y}}
\newcommand{\XX}{\underline{\mathbf{X}}}
\newcommand{\YY}{\underline{\mathbf{Y}}}
\newcommand{\X}{\mathbf{X}}
\newcommand{\zz}{\mathbf{z}}

\newcommand{\SO}{\mathsf{SO}}
\newcommand{\Sp}{\mathsf{Spin}}
\newcommand{\SL}{\operatorname{SL}}
\newcommand{\rank}{\operatorname{rk}}

\newcommand{\BA}{{\mathbb {A}}}

\newcommand{\BF}{{\mathbb {F}}}

\newcommand{\BN}{{\mathbb {N}}}

\newcommand{\BP}{{\mathbb {P}}}
\newcommand{\BQ}{{\mathbb {Q}}}
\newcommand{\BR}{{\mathbb {R}}}

\newcommand{\BZ}{{\mathbb {Z}}}
\newcommand{\CA}{{\mathcal {A}}}
\newcommand{\CB}{{\mathcal {B}}}

\newcommand{\CG}{{\mathcal {G}}}

\newcommand{\CP}{{\mathcal {P}}}
\newcommand{\CQ}{{\mathcal {Q}}}

\newcommand{\CU}{{\mathcal {U}}}
\newcommand{\CV}{{\mathcal {V}}}
\newcommand{\CW}{{\mathcal {W}}}
\newcommand{\CX}{{\mathcal {X}}}
\newcommand{\CY}{{\mathcal {Y}}}
\newcommand{\CZ}{{\mathcal {Z}}}
\newcommand{\RA}{{\mathbf {A}}}

\newcommand{\Aut}{{\mathrm{Aut}}}

\newcommand{\codim}{{\mathrm{codim}}}

\newcommand{\End}{{\mathrm{End}}}

\newcommand{\Gal}{{\mathrm{Gal}}}
\newcommand{\GL}{{\mathrm{GL}}}

\newcommand{\id}{{\mathrm{id}}}

\newcommand{\Pic}{\mathrm{Pic}}

\renewcommand{\mod}{\ \mathrm{mod}\ }

\newcommand{\Spec}{{\mathrm{Spec}}}

\newcommand{\sbt}{\subset}

\newcommand{\pr}{\operatorname{pr}}
\newcommand{\orbit}{\mathfrak{O}_{\adele}}
\newcommand{\m}{\operatorname{m}}

\theoremstyle{plain}

\newtheorem{theorem}{Theorem}[section]

\newtheorem{lemma}[theorem]{Lemma}
\newtheorem{corollary}[theorem]{Corollary}
\newtheorem{proposition}[theorem]{Proposition}

\theoremstyle{definition}
\newtheorem{definition}[theorem]{Definition}

\newtheorem{remark}[theorem]{Remark}

\newtheorem{question}[theorem]{\rm{\textbf{Question}}}

\newtheorem*{remark*}{Remark}
\newtheorem*{remarks*}{Remarks}
\newtheorem{hypothesis}[theorem]{Hypothesis}

\usepackage{geometry}
\usepackage{CJKutf8}

\begin{document}
\maketitle
		\begin{minipage}{\linewidth}
			\begin{flushright}
			\begin{CJK}{UTF8}{gkai}
	岂曰无衣，与子同袍。\\
	同气连枝，共盼春来。
			\end{CJK}
			\end{flushright}
			
		\end{minipage}
	
\begin{abstract}
	We establish the Hardy-Littlewood property (\emph{à la Borovoi-Rudnick}) for Zariski open subsets in affine quadrics of the form $q(x_1,\cdots,x_n)=m$, where $q$ is a non-degenerate integral quadratic form in $n\geqslant 3$ variables and $m$ is a non-zero integer. This gives asymptotic formulas for the density of integral points taking coprime polynomial values, which is a quantitative version of the arithmetic purity of strong approximation property off infinity for affine quadrics. 
\end{abstract}

\tableofcontents
\section{Introduction}
\subsection{Arithmetic purity of the Hardy-Littlewood property}
Let $X$ be a smooth geometrically integral variety over $\BQ$. Assume that $X$ satisfies strong approximation off the real place $\BR$. An open question first raised by Wittenberg (cf. \cite[\S2.7 Question 2.11]{Wittenberg}) asks whether all open subsets $U\subset X$ also satisfy strong approximation off $\BR$, whenever \begin{equation}\label{eq:condcodim}
	\codim_{X}(X\setminus U)\geqslant 2.  
\end{equation}
We say that such $X$ satisfies \emph{arithmetic purity of strong approximation} (\APSA\ for short) off $\BR$ (cf. \cite[Definition 1.2]{Cao-Huang}). As observed by Min\v{c}hev (cf. \cite[Proposition 2.6]{Borovoi-Rudnick}), the condition \eqref{eq:condcodim} guarantees that there is no cohomological or topological obstruction for $U$ to satisfy strong approximation. Recently in \cite{CLX,Cao-Huang}, Wittenberg's question was settled in the affirmative for a wide class of semisimple simply connected linear algebraic groups and their homogeneous spaces (with connected stabilizers). 
We refer to the references therein for an account of proven results towards this question.

In this article, we address an effective and statistical aspect of Wittenberg's question concerning the distribution of rational points in adelic spaces. Our starting point is the work of Borovoi-Rudnick \cite{Borovoi-Rudnick}. Now we assume that $X\subset\BA^n_\BQ$ is a smooth quasi-affine geometrically integral variety whose adelic space $X(\RA)$ is non-empty. Assume moreover that $X(\BR)$ has no compact connected components.
We equip the affine space $\BA^n(\BR)=\BR^n$ with an euclidean norm $\|\cdot\|$.
Let $B:=B_\infty\times B_f\subset X(\RA)$ be an adelic neighbourhood with $B_\infty\subset X(\BR)$ a real connected component and $B_f\subset X(\RA_f)$ a compact open subset.
For any $T>0$, consider the counting function
\begin{equation}\label{eq:NXT}
	N_X(B;T):=\#\{\XX\in X(\BQ)\cap B:\|\XX\|\leqslant T\}.
\end{equation}
We are interested in the varieties $X$ for which an asymptotic formula for $N_X(B;T)$ exists for any such $B$. Assume that $X$ satisfies strong approximation off $\BR$, then we expect that the leading constant should be a product of local densities (depending also on $\|\cdot\|$). However there also exist certain such varieties on which the integral Hasse principle or strong approximation can fail.   

To better put forth our qualitative characterization of these varieties, we further assume that $X$ is equipped with a fixed \emph{gauge form} $\omega_X$, i.e. a nowhere zero differential form of top degree. We associate to $\omega_X$ a \emph{normalised Tamagawa measure} $\m^X=\m_\infty^X\times\m_f^X$ on $X(\RA)$, where the real part $\m_\infty^X$ is a defined on $X(\BR)$, and the finite part $\m_f^X$ is defined on $X(\RA_f)$ (see $\S$\ref{se:normalisedTamagawameasure} for more details). Let $\delta_{X}:X(\RA)\to\BR_{\geqslant 0}$ be a locally constant not identically zero function, and let us define $$B_\infty(T):=\{\XX\in B_\infty:\|\XX\|\leqslant T\}.$$ 
We say that $X$ is a \emph{(relatively) Hardy-Littlewood} variety (with respect to the gauge form $\omega_X$) with density $\delta_{X}$, after Borovoi-Rudnick \cite[Definition 2.2]{Borovoi-Rudnick} (see also \cite[p. 143]{Duke-Rudnick-Sarnak}), if for any adelic neighbourhood $B=B_\infty\times B_f$ as before, we have $\m_\infty^X(B_\infty(T))\to\infty$ and \begin{equation}\label{eq:relHL}
	N_X(B;T)\sim \int_{B_\infty(T)\times B_f}\delta_X \operatorname{d}\m^X,  \quad T\to\infty.
\end{equation} 
A Hardy-Littlewood variety $X$ is called \emph{strongly Hardy-Littlewood} if $\delta_{X}\equiv 1$. 
That is, for any such $B$ as before, we have \begin{equation}\label{eq:strHL}
	N_X(B;T)\sim \int_{B_\infty(T)\times B_f} \operatorname{d}\m^X= \m_\infty^X(B_\infty(T))\m_f^X(B_f),\quad T\to\infty.
\end{equation} 
We may view the function $\delta_{X}$ as an effective measure of the failure of strong approximation on $X$. If $X(\BQ)=\varnothing$, then clearly $X$ is Hardy-Littlewood with density $\delta_{X}\equiv 0$. We shall henceforth focus on the case where $X(\BQ)\neq\varnothing$.
See notably works \cite{Duke-Rudnick-Sarnak,EM,EMS}, \cite[\S6]{Borovoi-Rudnick} and more recent ones \cite{Nevo-Sarnak,Gorodnik-Nevo,Browning-Gorodnik} for various examples of homogeneous spaces satisfying the Hardy-Littlewood property. \footnote{In all known examples, the counting function $N_X(B;T)$ behaves like $T^{d}(\log T)^e$ (depending on the embedding $X\hookrightarrow \BA^n_\BQ$), where $d>0,e\geqslant 0$ are rational numbers.}

The following is a natural extension and refinement of the Hardy-Littlewood property in the spirit of Wittenberg's question.
\begin{question}\label{q:countingpurity}
Assume that $X$ is a Hardy-Littlewood variety with density $\delta_{X}$. Are all open subsets $U\subset X$ also Hardy-Littlewood (with respect to the gauge form $\omega_X|_U$) with density $\delta_{X}|_U$, whenever $\codim_{X}(X\setminus U)\geqslant 2$?
\end{question}
\begin{definition}
	We say that $X$ satisfies \emph{arithmetic purity of the Hardy-Littlewood property} with density $\delta_{X}$, abbreviated as \APHL\ with density $\delta_{X}$, if $X$ is a Hardy-Littlewood variety with density $\delta_{X}$ and Question \ref{q:countingpurity} has a positive answer for $X$.

Similarly, we say that $X$ satisfies \emph{arithmetic purity of the strongly Hardy-Littlewood property}, abbreviated as \APSHL, if $X$ satisfies \APHL\ with density $\delta_{X}\equiv 1$.
\end{definition}
It is proved by Borovoi-Rudnick \cite[Proposition 2.5]{Borovoi-Rudnick} (resp. \cite[Proposition 2.4]{Borovoi-Rudnick}) that being relatively (resp. strongly) Hardy-Littlewood implies a weaker version of (resp. the usual) strong approximation for $X$. In particular, the condition \eqref{eq:condcodim} is necessary for $U$ to be Hardy-Littlewood. Moreover, we shall prove that (cf. Proposition \ref{prop:restrictionTamagawa} in $\S$\ref{se:normalisedTamagawameasure}) this condition guarantees that the restriction of the Tamagawa measure $\m^X$ on $U(\RA)$ is well-defined. To get an instructive idea, let us take $\CX$ an integral model of $X$ over $\BZ$, and let $\CZ$ be the Zariski closure of $Z:=X\setminus U$ in $\CX$ and let $\CU:=\CX\setminus\CZ$ be the integral model of $U$. For simplicity we assume that $\CU(\widehat{\BZ}):=\prod_{p<\infty} \CU(\BZ_p)\neq\varnothing$. In particular $\CU(\BF_p)\neq\varnothing$ for all $p$. Write $$\tau_p(\CU,\CX):=\frac{\#\CU(\BF_p)}{\#\CX(\BF_p)}.$$ We may interpret the quantity $\tau_p(\CU,\CX)$ as the probability that an integral point of $\CX$ specializes to a point in $\CU$ modulo $p$. Applying the Lang-Weil estimate (cf. \eqref{eq:Lang-WeilFp} \eqref{eq:Lang-WeilFpsmooth} \emph{infra}) to both $\CX$ and $\CZ$, we get
\begin{equation}\label{eq:LWcp}
	1-\tau_p(\CU,\CX)=\frac{\#\CZ(\BF_p)}{\#\CX(\BF_p)}=O\left(\frac{1}{p^{\codim_X(Z)}}\right),
\end{equation}
where the implied constant is uniform for any prime $p$. Roughly speaking, we calculate the finite part measure $\m_f^X|_U$ of $\CU(\widehat{\BZ})$ as being the limit of $\m_f^X(\prod_{p<M}\CU(\BZ_p)\times\prod_{p\geqslant M}\CX(\BZ_p))$ as $M\to \infty$. This turns out to require that infinite product over almost all $p$ of $\tau_p(\CU,\CX)$ is absolutely convergent, which is true whenever \eqref{eq:condcodim} holds thanks to \eqref{eq:LWcp}.
As removing a proper closed subset does not affect the real volume, we expect in particular that if the counting function $N_U(U(\BR)\times \CU(\widehat{\BZ});T)$ for $U$ is non-zero when $T\to\infty$, then \begin{equation}\label{eq:expNUNX}
	N_U(U(\BR)\times \CU(\widehat{\BZ});T)\sim c_{\CU,\CX} N_X(X(\BR)\times \CX(\widehat{\BZ});T),
\end{equation} where $c_{\CU,\CX}>0$ is related the infinite product of $\tau_p(\CU,\CX)$.
On the other hand,  we do not expect that, amongst other reasons (e.g., the existence of non-constant invertible functions, cf. \cite[Lemma 1.5.2]{Borovoi-Rudnick}), $N_U(U(\BR)\times \CU(\widehat{\BZ});T)$ could have the same magnitude of growth as $N_X(X(\BR)\times \CX(\widehat{\BZ});T)$ when $\codim_{X}(X\setminus U)=1$, as the infinite product over $\tau_p(\CU,\CX)$ would diverge to $0$. For related Schinzel-type conjectures however, see  \cite[(1.3)--(1.5)]{Nevo-Sarnak}.
 
Integral points of $\CX$ lying in such open subsets $\CU$ satisfy infinitely many congruence conditions, so that $\CU(\widehat{\BZ})\subset X(\RA)$ is neither open nor closed. To derive \eqref{eq:relHL} or \eqref{eq:strHL} for $U$, this is one difficulty  we need to get rid of. For instance, if $\CZ$ is defined by two regular functions $f,g\in\BZ[\CX]$, then $\CU(\BZ)$ consists of precisely the points $\XX\in \CX(\BZ)$ such that either $p\nmid f(\XX)$ or $p\nmid g(\XX)$ for any prime $p$, i.e.
$$\CU(\BZ)=\{\XX\in\CX(\BZ): \gcd(f(\XX),g(\XX))=1 \}.$$

\subsection{Results on \APHL}
Motivated by recent progress \cite{CLX,Cao-Huang} on \APSA\ for semisimple simply connected groups, the purpose of this article is to go beyond the affine spaces previously considered in \cite{Ekedahl,Poonen,Bhargava} and provide a positive answer to Question \ref{q:countingpurity} for affine quadrics, by developing a version of \emph{geometric sieve} for them. 
\subsubsection{Affine spaces}
The geometric sieve was first inaugurated by Ekedahl \cite{Ekedahl} when dealing with $X=\BA^n$. Further generalised by Poonen \cite[Theorem 3.1]{Poonen} and Bhargava \cite[\S3]{Bhargava}, this sieve method has demonstrated surprising applications on the density of square-free polynomial values in various circumstances. Their results indeed prove, on plugging-in the classical Chinese remainder theorem: \begin{theorem}[Ekedahl, Poonen, Bhargava]
	Affine spaces satisfy \APSHL.
\end{theorem}

\subsubsection{Affine quadrics}
Let $q(\xx)\in\BZ[x_1,\cdots,x_n] $ be a non-degenerate integral quadratic form in $n\geqslant 3$ variables. Suppose that $q(\xx)$ is indefinite (i.e. isotropic over $\BR$). For $m\in\BZ_{\neq 0}$, let us consider the $\BZ$-scheme
\begin{equation}\label{eq:integralmodelaffinequadric}
\mathcal{Q}:=(q(\xx)=m)\subset \BA^n_\BZ.
\end{equation}
The smooth geometrically integral $\BQ$-variety 
\begin{equation}\label{eq:affinequadric}
	\mathbf{Q}:=\mathcal{Q}\times_\BZ\BQ\subset\BA^n_\BQ
\end{equation} is called an \emph{affine quadric}.
In particular, $\CQ$ is an integral model of $\BFQ$. Recall that affine quadrics with $n\geqslant 4$ variables satisfy the integral Hasse principle and strong approximation, while there can be Brauer-Manin obstruction when $n=3$ (cf. \cite[\S5.6, \S5.8]{CT-Xu}).
We shall always assume in this article that $\BFQ(\BQ)\neq\varnothing$. We remark that rational points on an affine quadric form a \emph{thin} subset of $\BA^n(\BQ)$.

The variety $\BFQ$ is an affine symmetric space (cf. \cite[p. 59]{Borovoi-Rudnick}, \cite[p. 1045]{Browning-Gorodnik}) under the spin group $G:=\Sp_q$ (the universal double covering of $\SO_q$, with the standard almost faithful representation through $\SO_q$ in $\GL_{n,\BQ}$). Let $P\in\BFQ(\BQ)$ and let $H$ be its stabilizer. Then $H\cong \Sp_q|_{P^\perp}$ is a symmetric subgroup of $\Sp_q$, where $P^\perp$ is the orthogonal complement of $P$. The real locus $\BFQ(\BR)$ has no compact connected components. 
Since $n\geqslant 3$, the group $G$ is always semisimple and simply connected, so is $H$ if $n\geqslant 4$. Special attention is drawn to the case where $n=3$ because $H$ is isomorphic to a torus. It is anisotropic over $\BQ$ (hence has no non-trivial $\BQ$-characters) precisely when $-m\det(q)$ is not a square, a condition that we shall always assume and denote by $-m\det(q)\neq \square$ in the sequel.
The work of Borovoi-Rudnick \cite{Borovoi-Rudnick} proves that 
\begin{itemize}
	\item Affine quadrics with $n\geqslant 4$ variables are strongly Hardy-Littlewood (cf. \cite[Theorem 0.3]{Borovoi-Rudnick});
	\item Affine quadrics with $n=3$ variables are Hardy-Littlewood with a locally constant density function $\delta_{\BFQ}:\BFQ(\RA)\to \{0,2\}$ (cf. \cite[p. 1047-1048]{Browning-Gorodnik}, \cite[\S3]{Borovoi-Rudnick}, see also $\S$\ref{se:homogeneousspaces} \emph{infra}).
\end{itemize}
We recall the well-known asymptotic growth of integral points on affine quadrics (cf. e.g. \cite[(2.6)]{Liu-Sarnak}, \cite[(1.9)]{Duke-Rudnick-Sarnak} and \cite[p. 1047]{Browning-Gorodnik}). For every connected component $B_\infty\subset\BFQ(\BR)$,
\begin{equation}\label{eq:globalgrowthCQ}
	\#\{\XX\in \CQ(\BZ)\cap B_\infty:\|\XX\|\leqslant T\}\sim\int_{B_\infty(T)\times \CQ(\widehat{\BZ})}\delta_{\BFQ} \operatorname{d}\m^{\BFQ}\sim c_{\CQ,B_\infty}T^{n-2},
\end{equation} where $c_{\CQ,B_\infty} \geqslant 0$ depends on $\CQ,B_\infty$. 

Our main results are the following.
\begin{theorem}\label{thm:ngeq4}
	Let $\BFQ\subset\BA^n_\BQ$ be the affine quadric \eqref{eq:affinequadric}  with $n\geqslant 4$ variables. Then $\BFQ$ satisfies \APSHL.
\end{theorem}

\begin{theorem}\label{thm:n=3}
	Let $\BFQ\subset\BA^n_\BQ$ be the affine quadric \eqref{eq:affinequadric} with $n=3$ variables. Suppose that the form $q$ is anisotropic over $\BQ$ and that $-m\det(q)\neq\square$. Then $\BFQ$ satisfies \APHL\ with density $\delta_{\BFQ}$.
\end{theorem}
\subsection{A general strategy to achieve \APHL}\label{se:strategy}
In order to prove Theorems \ref{thm:ngeq4} and \ref{thm:n=3}, we develop a road-map, i.e. Theorem \ref{thm:HLtwistintegral}, which comes up with a sufficient criterion towards \APHL\ for general Hardy-Littlewood varieties. It consists of two hypotheses. For convenience of exposition, we fix an $X$ an affine Hardy-Littlewood variety with $\CX$ an affine integral model of $X$ over $\BZ$.

The first one is called \emph{effective equidistribution condition}. We require finer information on the remainder terms of the equivalences \eqref{eq:relHL} and \eqref{eq:strHL} when evaluated at a certain family of finite adelic neighbourhoods $B_f^\CX(\xi,l)$ (cf. \eqref{eq:BfX}, which we call \emph{congruence neighbourhoods} (Definition \ref{def:congneighbourhood})) associated to $l\geqslant 2$ and $\xi=(\xi_p)_{p\mid l}\in \prod_{p\mid l} X(\BQ_p)$. If $\xi\in\prod_{p\mid l}\CX(\BZ_{p})$, then $B_f^\CX(\xi,l)$ can be described in terms of congruence conditions: 
$$(B_\infty\times B_f^\CX(\xi,l))\cap \CX(\BZ)=\{\XX\in\CX(\BZ):\text{ for every }p\mid l,\XX\equiv \xi_p(\mod p^{v_p(l)}) \}.$$ 
The significant feature is that we require the error terms to have polynomial growth on the level of congruence $l$, and that we moreover require the implied constant to be uniform for both $l$ and $\xi$.
This formulation is motivated by recent works on counting integral points with congruence conditions due to Nevo-Sarnak \cite[\S3]{Nevo-Sarnak}, Gorodnik-Nevo \cite[\S6 \& \S7]{GW}, and Browning-Gorodnik \cite[\S2]{Browning-Gorodnik}, \emph{et al}. We utilise this hypothesis to deduce asymptotic formulas of integral points on $\CX$ of bounded height that specialise into an arbitrary subset of residues (cf. Proposition \ref{prop:HLintegralforsubset}), which has applications in two different directions. On the one hand, we can give an ``approximation'' of the Tamagawa measure restricted to a given open subset $U\subset X$ satisfying \eqref{eq:condcodim} (cf. Corollary \ref{co:appromainterm}). On the other hand, it turns out to give satisfactory control for integral points on $\CX$ coming from residues in a given closed subset of codimension at least two modulo certain primes which are not of polynomial growth (cf. Corollary \ref{co:uppercodim2}).

The second hypothesis is called \emph{geometric sieve condition} in the spirit of Ekedahl's work \cite{Ekedahl}, for which our formulation is closer to Bhargava's work \cite{Bhargava}. For any fixed closed $Z\subset X$ of codimension at least two, let $\CZ:=\overline{Z}\subset \CX$. Informally speaking, this hypothesis requires that integral points on $\CX$ arising from residues of $\CZ$ of arbitrarily large prime moduli are negligible. Establishing this hypothesis is often challenging in resolving Question \ref{q:countingpurity}, as we need to sieve out integral points satisfying infinitely many congruence conditions in a thin subset if the variety $X\subset\BA^n$ is not of degree one.

Then Theorem \ref{thm:HLtwistintegral} on the whole says that a Hardy-Littlewood variety $X$ satisfies \APHL, provided that sufficiently many integral models $\CX_a$ of $X$ constructed via ``rescaling the coordinates by $a\in\BN_{\neq 0}$'' (cf. \eqref{eq:CXa} in \S\ref{se:HLforintegralmodel} and \eqref{eq:CXaaffine} in $\S$\ref{se:applicationofequitoaffinevar}) all satisfy these two hypotheses.
Inspired by various results on counting lattice points in homogeneous spaces in a vast literature, we verify the effective equidistribution condition for certain \emph{nice} homogeneous spaces. This implies that to achieve \APHL\ for such varieties, it remains to verify the geometric sieve condition (cf. $\S$\ref{se:homogeneousspaces}).
\subsection{The geometric sieve for affine quadrics}
The technical core of our paper is the following extension of Ekedahl's geometric sieve \cite{Ekedahl} to affine quadrics, which seems to be the first example for affine varieties besides the affine spaces. 
\begin{theorem}\label{thm:geomsieve}
	Let $\BFQ$ ba an affine quadric satisfying the hypotheses of Theorems \ref{thm:ngeq4} and \ref{thm:n=3}. Fix an integral model $\CQ$ of $\BFQ$.
	For any closed subset $Z\subset \BFQ$ such that $\codim_{\BFQ}(Z)\geqslant 2$, let $\CZ$ be the Zariski closure of $Z$ in $\CQ$. Then uniformly for any $M\geqslant 2$, as $T\to\infty$, we have 
	$$\#\{\XX\in \mathcal{Q}(\BZ):\|\XX\|\leqslant T,\exists p\geqslant M, \XX\mod p\in \CZ(\BF_p)\}=O\left(T^{n-2}\left(\frac{1}{M^{\codim_{\BFQ}(Z)-1}\log M}+\frac{1}{\sqrt{\log T}}\right)\right),$$
	where the implied constant depends only on $\mathcal{Q}$ and $Z$.
\end{theorem}
The bound above being uniform for $M$, Theorem \ref{thm:geomsieve} thus furnishes a satisfactory control for integral points on $\CQ$ specialising into $\CZ$ modulo any prime larger than arbitrary $M=M(T)$ which grows to infinity, compared to the order of growth in \eqref{eq:globalgrowthCQ}.
We would like to mention the independent work of Browning and Heath-Brown \cite[Theorems 1.1--1.3]{Browning-HB}. Using a lattice point counting approach which technically differs from ours, they obtain similar results with a power saving of $T$ for the term involving $\frac{1}{\sqrt{\log T}}$, but only applied to projective quadrics (i.e., defined by homogeneous quadratic forms) with at least $5$ variables, whereas our result also applies to quadrics with fewer variables.

Our strategy of proving Theorem \ref{thm:geomsieve} unifies radically different ideas on estimating integral points on quadratic hypersurfaces. We start by breaking the prime moduli $[M,\infty]$ into three subfamilies, and then analyse each of them separately. Primes between $M$ and $T^\alpha$ with $\alpha>0$ small enough to be determined later are dealt with on utilising the effective equidistribution result mentioned in $\S$\ref{se:strategy} (cf. Theorem \ref{thm:primepoly}). 
Primes larger than $T^\alpha$ are divided into two parts: 
$$T^\alpha<p< T\quad \text{ and }\quad p\geqslant T.$$
We name them respectively ``intermediate primes'' and ``very large primes''.
For the intermediate primes (cf. Theorem \ref{th:intermediateprimes}), we appeal to uniform estimates for integral points of bounded height on quadrics in \cite[\S4 \S5]{Browning-Gorodnik} developed originally aiming at studying power-free polynomial values on affine quadrics. This slicing argument can be applied straightforward to the case $n\geqslant 4$ (cf. Proposition \ref{prop:intern4}). However it \emph{a priori} does not provide a desired power saving for the case $n=3$, as was also encountered in \cite[p. 1078]{Browning-Gorodnik}. To overcome this difficulty, we make essential use of the assumptions that $-m\det q$ is non-squared and that the form $q$ is $\BQ$-anisotropic in our argument. A consequence is that (cf. Lemma \ref{le:notsquare}) any slice contains no line or parabolic conic (i.e. affine plane conic of rank one) and thus contributes few integral points (cf. Proposition \ref{prop:intern3}). Note that these assumptions also appear in the results of Liu--Sarnak \cite{Liu-Sarnak}. 
For the treatment of very large primes (cf. Theorem \ref{th:largeprimes}), we make use of Bhargava's effective Ekedahl-type geometric sieve \cite{Bhargava}, and we develop a half-dimensional sieve for affine quadrics (cf. Theorem \ref{thm:halfdimsieve} below). These two sieve methods are matched together via a fibration argument (cf. $\S$\ref{se:er:geometricsieve}). 

\subsection{The half-dimensional sieve for affine quadrics}\label{se:halfsieve}
We establish the following auxiliary result on the density of quadratic polynomial values represented by a binary quadratic form. This is an application of the half-dimensional sieve due to Friedlander-Iwaniec \cite{Iwaniec} \cite{Friedlander-Iwaniec}, which may be of independent interest. 

\begin{theorem}\label{thm:halfdimsieve}
	Let $Q_1(\xx)\in\BZ[x_1,\cdots,x_L]$ be a quadratic polynomial in $L\geqslant 1$ variables, and let $Q_2(\yy)\in\BZ[y_1,y_2]$ be a binary positive-definite non-degenerate quadratic form. 
	\begin{itemize}
		\item If $L\geqslant 2$, assume that the $\BQ$-variety $(Q_1(\xx)=0)\subset \BA^L_\BQ$ is smooth;
		\item If $L=1$, assume that the $\BQ$-variety
		\begin{equation*}
		(Q_1(\xx)-Q_2(\yy)=0)\subset \BA^{3}_\BQ
		\end{equation*}
		is a (smooth) affine quadric, and that it has anisotropic stabilizer. 
	\end{itemize}
 Then
	$$\#\{\XX\in\BZ^{L}:\|\XX\|\leqslant T,\exists (u,v)\in\BZ^2,Q_1(\XX)=Q_2(u,v) \}=O\left(\frac{T^{L}}{\sqrt{\log T}}\right),$$
	where the implied constant depends only on $Q_1,Q_2$.
\end{theorem}
\begin{remark}\label{rmk:notsquare}
	Without the condition that the stabilizer be anisotropic, the estimate in Theorem \ref{thm:halfdimsieve} is false, as clearly seen from the example $x^2+1=y_1^2+y_2^2$. See also Remark \ref{rmk:isotropicsquare}.
\end{remark}

It would be interesting to ask whether Theorem \ref{thm:n=3} remains true for affine quadrics of dimension two with isotropic stabilizers. A different feature is that the singular series $\prod_{p<\infty}\frac{\#\CQ(\BF_p)}{p^{\dim \BFQ}}$ can diverge, and the order of magnitude of $N_\BFQ(\BFQ(\BR)\times \CQ(\widehat{\BZ});T)$ can be $T\log T$ instead of $T$ (cf. \cite[p. 146]{Duke-Rudnick-Sarnak}). We likewise ask whether condition that the form $q$ be $\BQ$-anisotropic could be dropped. On may run into the similar  pathological phenomenon as in the projective analogue \cite{Lindqvist}.

\subsection{Structure of the paper}
Section \ref{se:HLandequidist} is mostly served for technical preparations. We formulate the effective equidistribution condition and describe its applications. In Section \ref{se:applicationofequitoaffinevar} we prove Theorem \ref{thm:HLtwistintegral}, and then we deduce Theorems \ref{thm:ngeq4} and \ref{thm:n=3} as a consequence. Section \ref{se:errorterms} is entirely devoted to proving Theorem \ref{thm:geomsieve}. Theorem \ref{thm:halfdimsieve} is proved in Section \ref{se:polyresp}.  More layouts are sketched at the beginning of each section.

\subsection{Notation and conventions}\label{se:notation}
Given two real-valued functions $f$ and $g$ with $g$ non-negative, Vinogradov's symbol $f \ll g$ and Landau's symbol $f=O(g)$ both mean that there exists $C>0$ such that $|f|\leqslant Cg$. The dependence of $C$ on the variable and on $f,g$ will be specified explicitly. We use these two symbols interchangeably. We write $f\asymp g$ if $f\ll g$ and $g\ll f$ both hold. If $f,g$ are defined over the real numbers and with $g$ nowhere zero, the small ``$o$'' notation $f(x)=o(g(x))$ means that $\lim_{x\to \infty}\frac{f(x)}{g(x)}=0$. We write $f\sim g$ if $f-g=o(g)$, i.e., $\lim_{x\to \infty}\frac{f(x)}{g(x)}=1.$ In this article, all implied constants are allowed to depend on the embedding $X\hookrightarrow\BA^n_\BQ$ and the chosen euclidean norm $\|\cdot\|$.

The letter $p$ is always reserved for prime numbers, and $\varepsilon$ denotes an arbitrarily small positive parameter that can be rescaled by constant multiples. We write $p^k\| n$ for certain $k\in\BN,n\in\BN_{\neq 0}$ if $p^{k}\mid n$ and $p^{k+1}\nmid n$.
We denote the M\"{o}bius function by $\mu(\cdot)$. 
For every $l\in\BN_{\neq 0}$, we write $\BZ/l$ for the cokernel of the multiplication by $l$ morphism $\BZ\xrightarrow{\cdot l}\BZ$.
When stating a proposition, we sometimes call an integer $a\geqslant 2$ \emph{sufficiently divisible}, if there exists $a_0\in\BN_{\neq 0}$ such that $a_0\mid a$, and the dependency of the integer $a_0$ will be specified in each statement.

A \emph{variety} over $\BQ$ is an integral separated scheme of finite type over $\BQ$. 
For $X$ a $\BQ$-variety, an \emph{integral model} of $X$ over $\BZ$ is a faithfully flat of finite type and separated scheme $\CX$ over $\BZ$ 
endowed with an isomorphism $\CX\times_{\BZ}\BQ\cong X$ over $\BQ$. Such an integral model over $\BZ$ always exists by \cite[Proposition 2.5]{Liu-Xu}.
For each $x\in \CX$, we write $k(x)$ for the residue field of $x$.
We write $\codim_\CX(\CY)$ for the codimension of a subscheme $\CY$ in $\CX$.
For $\CW$ a scheme over $\BZ$, we write $\CW(\widehat{\BZ}):=\prod_{p<\infty} \CW(\BZ_{p})$.

We frequently use the following version of the Lang-Weil estimate, which follows from the Weil conjecture \cite[Theorem 7.7.1]{Poonena} ignoring finitely many primes.
Let $\CY$ be a separated reduced scheme of finite type of dimension $>0$ over $\BZ$. Let $\CY_{\BQ}:=\CY\times_{\BZ}\BQ$ be the generic fibre, and for any $p$, let $\CY_{\BF_p}:=\CY\times_{\BZ}\BF_p$. 
Since $\dim(\CY_{\BF_p})= \dim(\CY_{\BQ})$ for almost all $p$,  we have, uniformly for any $p$, \begin{equation}\label{eq:Lang-WeilFp}
\# \CY(\BF_p)=O(p^{\dim(\CY_\BQ)}).
\end{equation} If there exists $L_0\in\BN_{\neq 0}$ such that the scheme $\CY$ is smooth over $\BZ[1/L_0]$ with geometrically integral fibres (e.g. when the generic fibre $\CY_{\BQ}$ is smooth and geometrically integral), then uniformly for any $p$, \begin{equation}\label{eq:Lang-WeilFpsmooth}
\#\CY(\BF_p)=p^{\dim(\CY_\BQ)}+O(p^{\dim(\CY_\BQ)-\frac{1}{2}}).
\end{equation} All implied constants above depend only on $\CY$. In particular, by Hensel's lemma, for almost all $p$, $\CY(\BZ_p)\neq\varnothing$. 
Moreover, if $\CY$ is quasi-affine, then for any square-free integer $l\geqslant 2$ with $(L_0,l)=1$, it follows from \eqref{eq:Lang-WeilFpsmooth} that
\begin{equation}\label{eq:Lang-WeilZmodl}
\#\CY(\BZ/l)=\prod_{p\mid l}\#\CY(\BF_p)=\prod_{p\mid l}\left(p^{\dim(\CY_\BQ)}\left(1+O\left(p^{-\frac{1}{2}}\right)\right)\right)\ll_\varepsilon l^{\dim(\CY_\BQ)+\varepsilon}.
\end{equation}

\section{Hardy-Littlewood property and equidistribution}\label{se:HLandequidist}

 The Hardy-Littlewood property signifies that the growth of integral points is quantified by the Tamagawa measure on the adelic space. In \S\ref{se:normalisedTamagawameasure}  we first recall the definition of normalised Tamagawa measures on adelic spaces and their restrictions on open subsets. In \S\ref{se:HLforintegralmodel} we prove Proposition \ref{prop:sec2integralHL}, which shows that to deduce the \APHL\ property for $X$, it suffices to prove \eqref{eq:relHL} for congruence neighbourhoods with respect to a specific subfamily of integral models of $X$. In $\S$\ref{se:equierrterm} we prove Proposition \ref{prop:HLintegralforsubset}, giving explicit error terms for \eqref{eq:relHL} evaluated at a large family of adelic neighbourhoods. 
In \S\ref{se:applicationofequidistribution} we exhibit two different applications of Proposition \ref{prop:HLintegralforsubset}.
\subsection{Normalised Tamagawa measures}\label{se:normalisedTamagawameasure}
We refer to \cite[\S1.6]{Borovoi-Rudnick} and \cite[\S2]{Weil} for details.

Let $X$ be a smooth geometrically integral variety over $\BQ$ such that $X(\RA)\neq\varnothing$ (which we assume throughout this article).
A \emph{gauge form} $\omega_X$ on $X$ is a nowhere zero differential form of degree $\dim(X)$.
For any place $v$ of $\BQ$, the gauge form $\omega_X$ induces a measure $\m_v^X$ on $X(\BQ_v)$ (\cite[\S 2.2]{Weil}). 
Let $\CX$ be an integral model of $X$ over $\BZ$. By the Lang-Weil estimate \eqref{eq:Lang-WeilFpsmooth}, we can choose an integer $l$ (depending on $\CX$) such that $\prod_{p\nmid l,p<\infty}\CX(\BZ_p)\neq \varnothing$. 
According to \cite[Theorem 2.2.5]{Weil}, we have, for almost all finite places $v=p<\infty$, 
\begin{equation}\label{eq:weilresult}
\m_p^X(\CX(\BZ_p)):= \int_{\CX(\BZ_p)}\operatorname{d}\m^X_p= \frac{\#\CX(\BF_p)}{p^{\dim X}}. \footnote{Although it will not be used in the sequel, the real part $\m_\infty^X$ is closely related to the real Hardy-Littlewood density \emph{à la Siegel}, at least when $X$ is an affine complete intersection. See \cite[0.0.4]{Borovoi-Rudnick}.}
\end{equation}

A set of \emph{convergence factors} $(\lambda_v)$ for $(\m_v^X)$ is a set of strictly positive real numbers indexed by the places of $\BQ$ such that the infinite product \begin{equation}\label{eq:infiniteprodCX}
\prod_{p\nmid l,p<\infty}\lambda_p^{-1} \m^X_p(\CX(\BZ_p))
\end{equation} is absolutely convergent. 
The \emph{Tamagawa measure} on the adelic space $X(\RA)$  corresponding to $(\omega_X,(\lambda_v))$ is defined as (\cite[\S 2.3]{Weil}) \begin{equation}\label{eq:Tamagawames}
	\m^X:=\prod_v \lambda_v^{-1} \m_v^X.
\end{equation}
This Tamagawa measure $\m^X$ is \emph{normalized} if \begin{equation}\label{eq:normalisedTamagawa}
	\lambda_{\infty}\times \lim_{x\to \infty}\prod_{p\leq x}\lambda_p=1.
\end{equation}
The existence of the limit in \eqref{eq:normalisedTamagawa} (and hence the existence of a normalized Tamagawa measure) is equivalent to the convergence of the infinite product $\prod_{p\nmid l,p<\infty}\m^X_p(\CX(\BZ_p))$, and is independent of the choice of integral models $\CX$ of $X$ (cf. \cite[\S 1.6]{Borovoi-Rudnick}). 
Denote by \begin{equation}\label{eq:Tamagawafinite}
\m^X_f:=\lambda^{-1}_{\infty}\prod_{p<\infty} \lambda_p^{-1} \m_p^X
\end{equation} the finite part measure on $X(\RA_f)$. If $\m^X$ is normalised with respect to a chosen set of convergence factors, then a different choice of convergent factors satisfying \eqref{eq:normalisedTamagawa} induces the same measure as $\m^X$ (resp. $\m^X_f$) on $X(\RA)$ (resp. $X(\RA_f)$).

For any dense open subset $U\subset X$, for every place $v$, consider $\m_v^U:=\m_v^X|_U$ the restricted $v$-adic Tamagawa measure on $U_v(\BQ_v)$. For $v=\infty$, the set $(X\setminus U)(\BR)\subset X(\BR)$ is closed of lower dimension. Let $\phi:\CV\to X(\BR)$ be a local diffeomorphism, where $\CV\subset\BR^{\dim X}$ is an open real neighbourhood. Then for any bounded real neighbourhood $B_{\BR}\subset X(\BR)$, 
the Lebesgue measure of $\phi^{-1}(B_{\BR}\cap (X\setminus U)(\BR))$ is zero. Therefore  $\m_{\infty}^X(B_{\BR}\cap (X\setminus U)(\BR))=0$ and
\begin{equation}\label{eq:measurerealplace}
\m_{\infty}^X(B_{\BR})=\m_{\infty}^U (B_{\BR}\cap U(\BR)).
\end{equation}
The next proposition shows that, the restriction of $\m^X$ to $U$ is well-defined, provided $\codim_X(X\setminus U)\geqslant 2$.
\begin{proposition}\label{prop:restrictionTamagawa}
	Suppose that $\codim_X(X\setminus U)\geqslant 2$. Then $(\lambda_v)$ is a set of convergence factors for $(\m_v^U)$, and the product measure $\m^U:=\prod_{v}\lambda_v^{-1}\m_v^U$ is a normalized Tamagawa measure on $U(\RA)$ corresponding to $(\omega_X|_U, (\lambda_v))$. 
\end{proposition}
\begin{proof}
	Let us fix $\CX$ an integral model of $X$ over $\BZ$. Let $\CZ$ be the Zariski closure of $Z:=X\setminus U$ in $\CX$ and let $\CU:=\CX\setminus\CZ$ be the integral model of $U$. We fix $l^\prime\in\BN_{\neq 0}$ an integer such that $\CX,\CU$ are smooth over $\BZ[1/l^\prime]$ and $\prod_{p\nmid l^\prime,p<\infty}\CU(\BZ_p)\neq\varnothing$. 
	Take $M_0$ sufficiently large such that, for any $p\geqslant M_0$,  we have $p\nmid l^\prime$, and by (\ref{eq:weilresult}),
\begin{equation}\label{eq:mpXmpU}
		\m_p^X(\CX(\BZ_p))= \#\CX(\BF_p)p^{-\dim X},\quad \m_p^X(\CU(\BZ_p))=\m_p^U(\CU(\BZ_p))= \#\CU(\BF_p)p^{-\dim U}.
\end{equation}
	 The Lang-Weil estimates \eqref{eq:Lang-WeilFp} \eqref{eq:Lang-WeilFpsmooth} show that
	$$\frac{\#\CZ(\BF_p)}{\# \CU(\BF_p)}=O\left(\frac{1}{p^{\codim_{X}(Z)}}\right)$$ uniformly for all $p$. Therefore for any $p\geqslant M_0$,
	 \begin{equation}\label{eq:Lang-WeilCXCZ}
	 \frac{\lambda_{p}^{-1}\m_p^X(\CX(\BZ_p))}{\lambda_{p}^{-1}\m_p^X(\CU(\BZ_p))} 
	 =\frac{\#\CX(\BF_p)}{\# \CU(\BF_p)}=1+\frac{\#\CZ(\BF_p)}{\# \CU(\BF_p)}=1+O\left(\frac{1}{p^{\codim_{X}(Z)}}\right).
	 \end{equation}
	 This shows that the infinite product $$\prod_{p\nmid l^\prime,p<\infty}\frac{\lambda_{p}^{-1}\m_p^X(\CX(\BZ_p))}{\lambda_{p}^{-1}\m_p^X(\CU(\BZ_p))} $$ is absolutely convergent. Since by the construction of $(\lambda_v)$, \eqref{eq:infiniteprodCX} is absolutely convergent, we conclude that the infinite product $$\prod_{p\nmid l^\prime,p<\infty} \lambda_p^{-1}\m_p^U(\CU(\BZ_p))$$ is also absolutely convergent, which was to be shown.
\end{proof}
Normalized Tamagawa measures exist on many homogeneous spaces.
For instance, let $G$ be a semisimple and simply connected algebraic group, $H\sbt G$ be a connected reductive closed subgroup, and assume that $X\cong G/H$. Then by \cite[p. 24 Corollary]{Weil}, there exists a $G$-invariant gauge form on $X$.
In this case, we can construct a canonical Tamagawa measure $\m^X$  on $X(\RA)$ corresponding to a canonical choice of convergent factors $(\lambda_v)$ such that $\m^X$ is normalised (cf. \cite[\S 1.6.2]{Borovoi-Rudnick}).
If $H$ is semisimple, we choose $\lambda_v=1$ for all $v$, and in general, $\lambda_v$ are defined in terms of some Artin $L$-functions.

Throughout the rest of this section, We shall be working with varieties $X$ satisfying the following.
\begin{hypothesis}\label{hyp:good}
	The variety $X\subset \BA^n_{\BQ}$ is quasi-affine, smooth and geometrically integral, such that $X(\RA)\neq\varnothing$, that $X(\BR)$ has no compact connected components, and that $X$ is equipped with a normalized Tamagawa measure $\m^X=\m^X_{\infty}\times \m^X_f$ on $X(\RA)$.
\end{hypothesis}
\subsection{Hardy-Littlewood property for congruence neighbourhoods of integral models}\label{se:HLforintegralmodel}
Let $X$ be a variety satisfying Hypothesis \ref{hyp:good}.
Let $\CX\subset \BA^n_{\BZ}$ be an integral model of $X$ over $\BZ$.
 Let  $l\geqslant 2$ be an integer and $\xi:=(\xi_p)_{p|l}\in \prod_{p|l}X(\BQ_p)\subset\prod_{p\mid l}\BQ_p^n$ be a collection of local points. For each $p\mid l$, we define the $p$-adic neighbourhood $$B_p^{\CX}(\xi,l):=X(\BQ_p)\cap (\xi_p+p^{v_p(l)}\BZ^n_p)\subset \BQ^n_p.$$
We then define the finite adelic subset \begin{equation}\label{eq:BfX}
	B_f^{\CX}(\xi,l):=\prod_{p|l}B_p^{\CX}(\xi,l)\times \prod_{p\nmid l}\CX(\BZ_p)\subset X(\RA_f).
\end{equation}
\begin{definition}\label{def:congneighbourhood}
	We say that $B_f^{\CX}(\xi,l)$ is a compact congruence finite adelic neighbourhood of $\CX$, or a \emph{congruence neighbourhood} of $\CX$ for short, if it is non-empty and compact (by convention).
\end{definition}
Our goal is to show that (cf. Proposition \ref{prop:sec2integralHL} below),  the validity of formula \eqref{eq:relHL} for all such congruence neighbourhoods $B_f^{\CX}(\xi,l)$ ensures the Hardy-Littlewood property for $X$.

Note that $B_f^{\CX}(\xi,l)$ is non-empty if and only if $\prod_{p\nmid l}\CX(\BZ_p)\neq\varnothing$, and $B_f^{\CX}(\xi,l)$ is compact if and only if $B_p^{\CX}(\xi,l)$ is closed in $\xi_p+p^{v_p(l)}\BZ^n_p\subset\BA_{\BZ_{p}}^n$ (hence compact) for each $p|l$  because $\CX(\BZ_p)$ is compact for any $p$ (cf. the proof of \cite[Corollary 3.7]{Conrad}).
Any non-empty finite adelic subset $B_f^{\CX}(\xi,l)$ of $\CX$ of the form \eqref{eq:BfX} becomes a congruence neighbourhood upon replacing $l$ by some of its sufficiently large powers depending on $\CX$ and every $\xi_p$, and leaving $\xi$ unchanged.
Hence the family $\{B_f^{\CX}(\xi,l)\}$ consisting of all congruence neighbourhoods of $\CX$ forms a topological base of $X(\RA_f)$.

In many cases, congruence neighbourhoods can equivalently be defined via congruence conditions. For instance, when both $X,\CX$ are affine, i.e., both $X\subset\BA^n_\BQ$, $\CX\subset \BA^n_{\BZ}$ are closed, then for any $l\geqslant 2,\overline{\xi}\in \CX(\BZ/l)$, suppose that there exists $\xi$ a lift of $\overline{\xi}$ in $\prod_{p|l}\CX(\BZ_p)\subset \prod_{p|l}X(\BQ_p)$, we have  \begin{equation}\label{eq:congball}
	B_p^\CX(\xi,l)=\{\XX\in \CX(\BZ_p):\ \XX\equiv \overline{\xi} (\mod p^{v_p(l)})\}.
\end{equation} 
Each $B_p^\CX(\xi,l)$ is closed. Hence, if $B_f^{\CX}(\xi,l)$ is non-empty, it is a congruence neighbourhood. We will make use of this fact in \S\ref{se:homogeneousspaces}. 

Any $a\in \BN_{\neq 0}$ defines a morphism \begin{equation}\label{eq:phia}
	\begin{split}
		\phi_a: \BA^n_{\BZ}&\to \BA^n_{\BZ}\\ \xx&\mapsto a\cdot \xx,
	\end{split}
\end{equation} and $\phi_{a,\BQ}:\BA^n_\BQ\to\BA^n_\BQ$ is an isomorphism.
Let \begin{equation}\label{eq:Xa}
	X_a:=\phi_{a,\BQ}(X)\subset \BA^n_{\BQ}.
\end{equation} Then $X_a$ is isomorphic to $X$, and is equipped with the normalised Tamagawa measure $\m^{X_a}:=(\phi_{a,\BQ})_*\m^X$. 
Now we define an integral model $\CX_a$ of $X_a$.
Let $\overline{X}\subset \BA^n_{\BQ}$, $\overline{\CX}\subset \BA^n_{\BZ}$ be their corresponding Zariski closures, 
and assume that $\overline{\CX}$ is defined by a family of polynomials $f_1,\cdots,f_r\in \BZ[x_1,\cdots,x_n].$
 Let $\overline{\CX_a}\subset \BA^n_{\BZ}$ be the affine scheme defined by the polynomials $$a^{\deg f_1}f_1\left(\frac{1}{a}x_1,\cdots,\frac{1}{a}x_n\right),\cdots,a^{\deg f_r}f_r\left(\frac{1}{a}x_1,\cdots,\frac{1}{a}x_n\right)\in\BZ[x_1,\cdots,x_n],$$
 and let \begin{equation}\label{eq:CXa}
 	\CX_a:=\overline{\CX_a}\setminus \overline {\phi_{a,\BQ}(\overline{X}\setminus X)} \subset \BA^n_{\BZ}. \footnote{The inner bar denotes the Zariski closure in $\BA^n_\BQ$, and the outer bar denotes the closure in $\BA^n_{\BZ}$.}
 \end{equation}
Then $\CX_a$ is an integral model of $X_a$ and, in fact, $\overline{\CX_a}$ is the Zariski closure of $X_a$ in $ \BA^n_{\BZ}$. If $\CX=\overline{\CX}\setminus \overline{(\overline{X}\setminus X)}$ (e.g. when both $X,\CX$ are affine), then $\phi_a$ induces an isomorphism $ \CX\times_{\BZ}\BZ[1/a] \cong \CX_a\times_{\BZ}\BZ[1/a]$.

Recall as in the introduction that we equip the affine space $\BA^n(\BR)=\BR^n$ with an euclidean norm $\|\cdot\|$ in order to define the counting function \eqref{eq:NXT}. 
\begin{proposition}\label{prop:sec2integralHL}
Assume that $X$ satisfies Hypothesis \ref{hyp:good}, and let $\delta_{X}:X(\RA)\to\BR_{\geqslant 0}$ be a locally constant, not identically zero function. Then the following properties are equivalent.

(i) The variety $X$ is Hardy-Littlewood with density $\delta_{X}$;

(ii) For any quasi-affine integral model $\CX\subset \BA^n_{\BZ}$ of $X$, for any connected component $B_\infty\subset X(\BR)$ and for any congruence neighbourhood $ B_f^{\CX}(\xi,l)$ of $\CX$, we have \begin{equation}\label{eq:sec2formulaintegralHL}
	N_{X}(B_{\infty}\times B_f^{\CX}(\xi,l);T) \sim \int_{B_\infty(T)\times B_f^{\CX}(\xi,l)}\delta_X \operatorname{d}\m^X, \quad T\to\infty;
\end{equation}

(iii) There exists a quasi-affine integral model $\CX\subset \BA^n_{\BZ}$ of $X$ such that, for any connected component $B_\infty\subset X(\BR)$ and for any congruence neighbourhood $ B_f^{\CX}(\xi,l)$ of $\CX$, the formula (\ref{eq:sec2formulaintegralHL}) holds ;

(iv) There exists a quasi-affine integral model $\CX\subset \BA^n_{\BZ}$ of $X$ such that, for any connected component $B_\infty\subset X(\BR)$, and for any sufficiently divisible $a\in \BN_{\neq 0}$ (depending on $\CX$) and any congruence neighbourhood $ B_f^{\CX_a}(\xi,l)$ of $\CX_a$ with $B_f^{\CX_a}(\xi,l)\subset \overline{\CX_a}(\widehat{\BZ})$, we have \begin{equation}\label{eq:sec2formulaintegralHLtwista}
	N_{X_a}(\phi_{a,\BR}(B_{\infty})\times B_f^{\CX_a}(\xi,l);T)\sim  \int_{\phi_{a,\BR}(B_\infty(T))\times B_f^{\CX_a}(\xi,l)}\delta_{X_a} \operatorname{d}\m^{X_a},
\end{equation} where $\delta_{X_a}$ is the locally constant function $\delta_{X}\circ \phi_a^{-1}:X_a(\RA)\to \BR_{\geqslant 0}$.
\end{proposition}

\begin{proof}
The implications ``(ii)$\Rightarrow$(iii)'' and ``(ii)$\Rightarrow$(iv)'' are evident, and ``(i)$\Rightarrow$(ii)'' just follows from definition. 

For ``(iii)$\Rightarrow$(i)'', we need to show: for any connected component $B_\infty\subset X(\BR)$ and any compact open subset $B_f\subset X(\RA_f)$, the formula (\ref{eq:relHL}) holds.
Let $\CX$ be as in the assumption satisfying \eqref{eq:sec2formulaintegralHL}. Since $B_f$ is compact, there exist finitely many congruence neighbourhoods $ B_f^{\CX}(\xi_i,l_i),1\leqslant i\leqslant r$ of $\CX$ such that $B_f=\cup_{i=1}^r B_f^{\CX}(\xi_i,l_i)$.
We can choose $l_0\in \BN_{\neq 0}$ (depending on $\CX$ and $B_f$) with $\prod_{i=1}^r l_i\mid l_0$ such that each $ B_f^{\CX}(\xi_i,l_i)$ is covered by finitely many congruence neighbourhoods $B_f^{\CX}(\xi_j,l_0)$ of $\CX$. Since any two such neighbourhoods $ B_f^{\CX}(\xi_j,l_0)$ are either equal or disjoint, therefore we conclude that there exist finitely many congruence neighbourhoods $B_f^{\CX}(\xi_k,l_0),1\leqslant k\leqslant r^\prime$ such that
$B_f=\coprod_{k=1}^{r^\prime} B_f^{\CX}(\xi_k,l_0)$.
Hence, assuming that \eqref{eq:sec2formulaintegralHL} holds for all adelic neighbourhoods $B_\infty\times B_f^{\CX}(\xi_k,l_0),1\leqslant k\leqslant r^\prime$, then as $T\to\infty$, 
\begin{align*}
	N_X(B_\infty\times B_f;T)&=\sum_{k=1}^{r^\prime} N_X(B_\infty\times B_f^{\CX}(\xi_k,l_0);T)\\ &\sim \sum_{k=1}^{r^\prime}\int_{B_\infty(T)\times B_f^{\CX}(\xi_k,l_0)}\delta_X \operatorname{d}\m^X\\ &= \int_{B_\infty(T)\times B_f}\delta_X \operatorname{d}\m^X.
\end{align*}
This proves \eqref{eq:relHL}, whence $X$ is Hardy-Littlewood with density $\delta_{X}$.

Now we prove ``(iv)$\Rightarrow$(iii)''. Let $\CX$ be as in the assumption.
We may assume $\CX=\overline{\CX}\setminus \overline{(\overline{X}\setminus X)}$, which does not affect the definition of $\CX_a$ \eqref{eq:CXa}. Let $a_0\in\BN_{\neq 0}$ (depending on $\CX$) be such that the assumption of (iv) holds for every $\CX_a$ with $a_0\mid a$.
Now for any congruence neighbourhood $ B_f^{\CX}(\xi,l)$ of $\CX$,
we can find $a_l\in \BN_{\neq 0}$ (depending on $B_f^{\CX}(\xi,l)$) with $l\mid a_l$ and $a_0\mid a_l$ such that $\phi_{a_l}(\xi)\in \prod_{p|l}\overline{\CX_{a_l}}(\BZ_p)$. Note in particular $\prod_{p\nmid l}\overline{\CX_{a_l}}(\BZ_p)\neq\varnothing$ by the definition of $ B_f^{\CX}(\xi,l)$. Then it is clear that $\overline{\CX_{a_l}}(\widehat{\BZ})\neq\varnothing$, $B_f^{\CX_{a_l}}(\phi_{a_l}(\xi),a_ll)\subset\overline{\CX_{a_l}}(\widehat{\BZ})$ is a congruence neighbourhood of $\CX_{a_l}$, and $\phi_{a_l}$ induces an isomorphism $B_\infty\times B^{\CX}_f(\xi,l)\cong \phi_{a_l,\BR}(B_{\infty})\times B^{\CX_{a_l}}_f(\phi_{a_l}(\xi),a_ll)$. Therefore since $a_0\mid a_l$, applying \eqref{eq:sec2formulaintegralHLtwista} to $B^{\CX_{a_l}}_f(\phi_{a_l}(\xi),a_ll)$, as $T\to\infty$, we have
\begin{align*}
	N_{X}(B_{\infty}\times B_f^{\CX}(\xi,l);T)&=N_{X_{a_l}}(\phi_{a_l,\BR}(B_{\infty})\times B_f^{\CX_{a_l}}(\phi_{a_l}(\xi),a_ll);a_lT)\\ &\sim \int_{(\phi_{a_l,\BR}(B_{\infty}))(a_lT)\times B_f^{\CX_{a_l}}(\phi_{a_l}(\xi),a_ll)}\delta_{X_{a_l}} \operatorname{d}\m^{X_{a_l}}\\ &=\int_{B_\infty(T)\times B_f^{\CX}(\xi,l)}\delta_X \operatorname{d}\m^X.
\end{align*}
This proves \eqref{eq:sec2formulaintegralHL} for the integral model $\CX$.
\end{proof}

We end this subsection by the following lemma which will be helpful later.
\begin{lemma}\label{le:factorisationdelta}
	Assume that $X$ satisfies Hypothesis \ref{hyp:good}. Let $\delta_{X}:X(\RA)\to\BR_{\geqslant 0}$ be a locally constant function. Then we have:
	 \begin{enumerate}[label=(\roman*)]
	 	\item For any connected component $B_{\infty}\subset X(\BR)$, $\delta_X|_{B_{\infty}\times X(\RA_f)}$ factors through a locally constant function $\delta_{B_{\infty}}:X(\RA_f)\to \BR_{\geqslant 0}$. In particular, for any compact open $B_f\subset X(\RA_f)$, we have \begin{equation}\label{eq:deltaBinfty}
	 		\int_{B_\infty(T)\times B_f}\delta_X \operatorname{d}\m^X=\m_{\infty}^X(B_{\infty}(T))\int_{B_f}\delta_{B_\infty} \operatorname{d}\m_f^X.
	 	\end{equation}
	 	\item Keeping the notation in (i). For any $l_0\geqslant 2$ and any compact open subset $B_0\sbt \prod_{p\mid l_0}X(\BQ_p)$, and for any quasi-affine integral model $\CX$ of $X$, on writing \begin{equation}\label{eq:Bf0X}
	 		B_{f,0}^\CX:=B_0\times \prod_{p\nmid l_0}\CX(\BZ_p)\subset X(\RA_f),
	 	\end{equation} there exists $l_1\in \BN_{\neq 0}$ depending only on $B_\infty,B_0$ and $\CX$ with $l_0\mid l_1$ such that, $\delta_{B_{\infty}}|_{B_{f,0}^\CX}$ factors through the projection $B_{f,0}^\CX\to B_0\times \prod_{p\mid l_1, p\nmid l_0}\CX(\BZ_p)$.  
	 \end{enumerate}
\end{lemma}
\begin{proof}
	(i) is clear because $\delta_{X}$ is locally constant and $B_{\infty}$ is connected.
	
	We now prove (ii). By the same arguments for the proof of ``(iii)$\Rightarrow$(i)'' of Proposition \ref{prop:sec2integralHL}, 
	there exists an $l_1\in \BN_{\neq 0}$ depending only on $B_\infty$ and $\CX$ with $l_0\mid l_1$ such that, the compact set $B_0\times \prod_{p\nmid l_0}\CX(\BZ_p)$ is a finite union of open subsets of the form $B_f^{\CX}(\xi_k,l_1)$, and the locally constant function $\delta_{B_\infty}$ obtained in (i) is constant on each $B_f^{\CX}(\xi_k,l_1)$.
	Therefore $\delta_{B_{\infty}}|_{B_{f,0}^\CX}$ factors through $ B_0\times \prod_{p\mid l_1, p\nmid l_0}\CX(\BZ_p)$, whence the statement of (ii).
\end{proof}
\subsection{Effective equidistribution with error terms}\label{se:equierrterm}
Let $X$ be a smooth geometrically integral quasi-affine variety over $\BQ$ with a fixed integral model $\CX$ over $\BZ$. Fix $l_0\in \BN_{\neq 0}$ an integer (depending on $\CX$) such that $\CX\times_{\BZ}\BZ[1/l_0]$ is smooth with geometrically integral fibres over $\BZ[1/l_0]$, and that $\prod_{p\nmid l_0}\CX(\BZ_p)\neq\varnothing$. For any integer $l\geqslant 2$ with $\gcd(l,l_0)=1$, the residue map
\begin{equation}\label{eq:defPsi}
	\Psi_l:=\prod_{p\mid l}\Psi_p: \prod_{p\mid l}\CX(\BZ_p) \to \prod_{p\mid l} \CX(\BZ/p^{v_p(l)}) \cong \CX(\BZ/l)
\end{equation}
is surjective by Hensel's Lemma.
Therefore, for any subset $S\subset \CX(\BZ/l)$, the set $\Psi_l^{-1}(S)\subset \prod_{p\mid l}\CX(\BZ_p)$ is compact. 

In the following proposition, we shall consider \eqref{eq:uniformlyHLintegral} as an effective version of the asymptotic formula (\ref{eq:sec2formulaintegralHL}) with an explicit uniform error term.
\begin{proposition}\label{prop:HLintegralforsubset}
Assume that $X$ satisfies Hypothesis \ref{hyp:good}. Let $\delta_{X}:X(\RA)\to\BR_{\geqslant 0}$ be a locally constant, not identically zero function. Assume that for $B_0\sbt \prod_{p\mid l_0}X(\BQ_p)$ a fixed compact open subset and $B_\infty\subset X(\BR)$ a fixed connected component, 
there exist $\sigma_\CX\geqslant 0$ and $0<\beta_\CX<1$ such that for any congruence neighbourhood $B_f^{\CX}(\xi,l)\subset B_{f,0}^\CX$ (cf. \eqref{eq:Bf0X}), we have
\begin{equation}\label{eq:uniformlyHLintegral}
	N_{X}(B_{\infty}\times B_f^{\CX}(\xi,l);T) = \int_{B_\infty(T)\times B_f^{\CX}(\xi,l)}\delta_X \operatorname{d}\m^X+O(l^{\sigma_\CX}\m_{\infty}^X(B_\infty(T))^{1-\beta_\CX}),
\end{equation}
where the implied constant of (\ref{eq:uniformlyHLintegral}) may depend on $B_{\infty},\CX,B_0$ but it is independent of $l$ and $\xi$. 
 Then for any integer $l\geqslant 2$ with $\gcd(l,l_0)=1$ and any subset $S\subset \CX(\BZ/l)$, we have
\begin{equation}\label{eq:uniformlyHLintegral2}
	N_{X}(B_{\infty}\times B_{f,0}^{\CX}(S,l);T) = \int_{B_\infty(T)\times B_{f,0}^{\CX}(S,l)}\delta_X \operatorname{d}\m^X+O(\#S \cdot l^{\sigma_\CX}\cdot\m_{\infty}^X(B_\infty(T))^{1-\beta_\CX})
\end{equation}
where we write \begin{equation}\label{eq:Bf0}
	B_{f,0}^{\CX}(S,l):=B_0\times \Psi_l^{-1}(S)\times \prod_{p\nmid l_0l} \CX(\BZ_p)\subset X(\RA_f),
\end{equation} and the implied constant may depend on $B_{\infty},\CX,B_0$, but it is independent of $l$ and $S$.
\end{proposition}

\begin{proof}
Since $B_0$ is compact, as in the proof of Proposition \ref{prop:sec2integralHL} (see also the discussion after Definition \ref{def:congneighbourhood}), we can choose an integer $l_1$ (depending only on $B_0,\CX$) with the same prime factors as $l_0$ and finitely many congruence neighbourhoods $(B_f^{\CX}(\xi_i,l_1))_{i\in I} $ of $\CX$ such that 
\begin{equation}\label{eq:B0}
	B_0=\coprod_{i\in I} (\prod_{p\mid l_0}B^{\CX}_p(\xi_i,l_1)).
\end{equation}

Now let us fix $l\geqslant 2$ with $\gcd(l,l_0)=1$ and $S\subset \CX(\BZ/l)$. For every $\overline{s}\in S$, we fix a lift $s\in \Psi_l^{-1}(\{\overline{s}\})\subset \prod_{p\mid l}\CX(\BZ_{p})$. Then for such $s$ and any $i\in I$, we have $\xi_i\times s \in\prod_{p\mid l_1l}X(\BQ_p)$, and $B_f^{\CX}(\xi_i\times s,l_1l)$ is a congruence neighbourhood of $\CX$.
Moreover, thanks to \eqref{eq:B0}, we have
$$ B_{f,0}^{\CX}(S,l)=\coprod_{i\in I,\overline{s}\in S} B^{\CX}_f(\xi_i\times s,l_1l). $$
Hence the formula (\ref{eq:uniformlyHLintegral}) implies
\begin{align*}
	N_{X}(B_{\infty}\times B_{f,0}^{\CX}(S,l);T)&=\sum_{i\in I,\overline{s}\in S}\left(\int_{B_\infty(T)\times B_f^{\CX}(\xi_i\times s,l_1l)}\delta_X \operatorname{d}\m^X+O((l_1l)^{\sigma_\CX}\m_{\infty}^X(B_\infty(T))^{1-\beta_\CX})\right)\\ &=
	\int_{B_\infty(T)\times B_{f,0}^{\CX}(S,l)}\delta_X \operatorname{d}\m^X
	+ O(\# I\cdot \# S\cdot (l_1 l)^{\sigma_\CX}\m_{\infty}^X(B_\infty(T))^{1-\beta_\CX}).
\end{align*}
Since $\#I, l_1$ only depend on $B_0,\CX$, the desired formula \eqref{eq:uniformlyHLintegral2} then follows, where the implied constant is only allowed to depend on $B_\infty,\CX,B_0$.
\end{proof}

\subsection{Two applications of effective equidistribution}\label{se:applicationofequidistribution}
Let $Z\subset X$ be a proper closed subset, and let $U:=X\setminus Z$ be the open subset. Let $\CZ$ be the Zariski closure of $Z$ in $\CX$ as an integral model of $Z$ and let $\CU:=\CX\setminus\CZ$ be the integral model of $U$. By Proposition \ref{prop:restrictionTamagawa}, the restriction $\m^U:=\m^X|_U$ is well-defined as a normalised Tamagawa measure on $U(\RA)$.

Our first application of Proposition \ref{prop:HLintegralforsubset} consists in comparing the counting function of integral points on $\CX$ lying in $\CU$ modulo any $p<M,p\nmid l_0$, with the Tamagawa measure of the corresponding adelic neighbourhood of $U(\RA)$, when $M$ is sufficiently large. We shall make use of this result in the deduction of \APHL\ in Theorem \ref{thm:HLtwistintegral}.

For any integer $l\geqslant 2$ with $\gcd(l,l_0)=1$, recall \eqref{eq:defPsi} and take $\CU(\BZ/l)\subset\CX(\BZ/l)$. 
Then \begin{equation}\label{eq:Psil-1}
	\Psi_l^{-1}(\CU(\BZ/l))=\prod_{p\mid l}\CU(\BZ_p),
\end{equation} because $\Psi_p^{-1}(\CU(\BZ/p^{v_p(l)}))=\CU(\BZ_p)$ for every $p\mid l$. \footnote{Indeed, any $\BZ/p^{v_p(l)}$-point of $\CU$, say $P$, can be lifted to a $\BZ_p$-point $\widetilde{P}$ of $\CX$, by Hensel's lemma. Let $\overline{P}$ be the image of $P$ in $\CU(\BF_p)$. Then since the open set $\CU$ contains the closed point $\overline{P}$ of $\Spec(\BZ_p)$, $\CU$ must also contain the generic point of $\Spec(\BZ_{p})$, and therefore $\widetilde{P}\in\CU$.} 
For any $M>2$, we note
\begin{equation}\label{eq:PMl0}
	\mathfrak{P}_{M,l_0}:=\prod_{p\nmid l_0, p<M} p.
\end{equation}
Hence with the notation \eqref{eq:Bf0},
\begin{equation}\label{eq:defiBCUMl0}
	\begin{split}
		B^\CX_{f,0}(\CU(\BZ/\mathfrak{P}_{M,l_0}),\mathfrak{P}_{M,l_0})&=B_0\times \Psi_l^{-1}(\CU(\BZ/\mathfrak{P}_{M,l_0}))\times \prod_{p\geqslant M}\CX(\BZ_p)\\ &=B_0\times\prod_{p\nmid l_0,p<M}\CU(\BZ_p)\times\prod_{p\geqslant M}\CX(\BZ_p).
	\end{split} 
\end{equation}
\begin{corollary}\label{co:appromainterm}
		Under the hypothesis of Proposition \ref{prop:HLintegralforsubset}, we assume moreover $B_0\subset \prod_{p\mid l_0}U(\BQ_p)$, $\prod_{p\nmid l_0}\CU(\BZ_p)\neq\varnothing$ and $\codim_X(Z)\geqslant 2$.
		Then, for any $M>2$ sufficiently large, we have
	\begin{align*}
		&N_X(B_\infty\times B_{f,0}^{\CX}(\CU(\BZ/\mathfrak{P}_{M,l_0}),\mathfrak{P}_{M,l_0});T)\\
		=&\int_{B_{\infty}^U(T)\times B_{f,0}^\CU}\delta_X|_U \operatorname{d}\m^U+O_\varepsilon\left(	\mathfrak{P}_{M,l_0}^{\sigma_\CX+\dim(X)+\varepsilon}\m^X_{\infty}(B_\infty(T))^{1-\beta_\CX}+\frac{\m_{\infty}^X(B_\infty(T))}{M^{\codim_{X}(Z)-1}}\right),
	\end{align*}
	where $B_{f,0}^\CU\subset U(\RA_f)$ is defined analogously to \eqref{eq:Bf0X} in terms of $\CU$, \begin{equation}\label{eq:BUT}
		B_{\infty}^U(T):=B_{\infty}(T)\cap U(\BR),
	\end{equation} and the implied constant may depend on $B_{\infty},B_0,\CU,\CX$ but it is independent of $M$.
\end{corollary}
\begin{proof}
	On applying Lemma \ref{le:factorisationdelta} (i) to $\delta_{X}$, let $\delta_{B_{\infty}}:X(\RA_f)\to\BR_{\geqslant 0}$ be the factorisation of $\delta_{X}$. Using \eqref{eq:deltaBinfty}, we then consider $$I_{\CX}(\CU(\BZ/\mathfrak{P}_{M,l_0})):=\int_{B_\infty(T)\times B_{f,0}^{\CX}(\CU(\BZ/\mathfrak{P}_{M,l_0}),\mathfrak{P}_{M,l_0})}\delta_X \operatorname{d}\m^X=\m_{\infty}^X(B_{\infty}(T))\int_{B_{f,0}^{\CX}(\CU(\BZ/\mathfrak{P}_{M,l_0}),\mathfrak{P}_{M,l_0})}\delta_{B_\infty} \operatorname{d}\m_f^X,$$ and
\begin{equation}\label{eq:IU}
		I_{\CU}:=\int_{B_{\infty}^U(T)\times B_{f,0}^\CU}\delta_X|_U \operatorname{d}\m^U=\m_\infty^U(B^U_\infty(T))\int_{B_{f,0}^\CU}\delta_{B_{\infty}}|_U \operatorname{d}\m^U_f. 
\end{equation} The formula \eqref{eq:measurerealplace} shows that $$\m_{\infty}^X(B_{\infty}(T))=\m_{\infty}^U(B_{\infty}^U(T)).$$
Our goal is to compare the finite part integral in $I_{\CX}(\CU(\BZ/\mathfrak{P}_{M,l_0}))$ with the one in $I_{\CU}$.
	
We apply Lemma \ref{le:factorisationdelta} (ii) to the locally constant function $\delta_{B_\infty}$, and let $l_1\geqslant 2$ (depending only on $B_{\infty},B_0,\CX$) be an integer satisfying (ii) for $B_\infty$ and $B_0$. We take $M_1>2$ (depending only on $B_{\infty},B_0,\CX$) such that $p\geqslant M_1\Rightarrow p\nmid l_1$ (for example, we can take $M_1=l_1+1$).
Then, for any $M\geqslant M_1$, when restricted to $B_{f,0}^\CX$, $\delta_{B_{\infty}}$ factors through a locally constant function 
\begin{equation}\label{eq:deltaBM}
	\delta_{B_\infty}^M:B_{0,M}^\CX:= B_0\times \prod_{p<M,p\nmid l_0}\CX(\BZ_p)\to \BR_{\geqslant 0},
\end{equation}
i.e., $\delta_{B_{\infty}}|_{B_{f,0}^\CX}$ is the composition of $\delta_{B_\infty}^M$ with the projection $B_{f,0}^\CX\to B_{0,M}^\CX.$
Therefore thanks to \eqref{eq:defiBCUMl0}, on defining $B_{0,M}^\CU:=B_0\times \prod_{p\nmid l_0, p<M}\CU(\BZ_p)\subset B_{0,M}^\CX $, we have
$$B_{f,0}^{\CX}(\CU(\BZ/\mathfrak{P}_{M,l_0}),\mathfrak{P}_{M,l_0})=B_{0,M}^\CU\times \prod_{p\geqslant M} \CX(\BZ_p),\quad B_{f,0}^\CU=B_{0,M}^\CU\times \prod_{p\geqslant M} \CU(\BZ_p).$$
On defining moreover the measures $$\m^X_M:=\lambda^{-1}_{\infty}\prod_{p<M} \lambda_p^{-1} \m_p^X, \quad\text{and} \quad \m^U_M:=\m^X_M|_U,$$ on respectively $\prod_{p<M}X(\BQ_p)$ and $\prod_{p<M}U(\BQ_p)$, and using \eqref{eq:deltaBM}, we can now compute:
$$I_{\CX}(\CU(\BZ/\mathfrak{P}_{M,l_0}))=\m_{\infty}^X(B_{\infty}(T))  \left( \int_{B_{0,M}^\CU}\delta_{B_\infty}^M \operatorname{d}\m^X_M \right) \left( \prod_{p\geqslant M} \lambda_p^{-1}\m_p^X(\CX(\BZ_p))\right),$$ and
$$I_{\CU}=\m_{\infty}^U(B_{\infty}^U(T))\left(\int_{B_{0,M}^\CU}\delta_{B_\infty}^M|_U  \operatorname{d}\m^U_M\right) \left(\prod_{p\geqslant M} \lambda_p^{-1}\m_p^U(\CU(\BZ_p)) \right).  $$
By definition, 
\begin{equation}\label{eq:integralforMset}
\int_{B_{0,M}^\CU} \delta_{B_\infty}^M \operatorname{d}\m^X_M= \int_{B_{0,M}^\CU} \delta_{B_\infty}^M|_U  \operatorname{d}\m^U_M.
\end{equation}
It remains to compare the infinite products. According to \eqref{eq:Lang-WeilCXCZ} as in the proof of Proposition \ref{prop:restrictionTamagawa}, if $M>M_0$, on writing $c_p:=\frac{\#\CZ(\BF_p)}{\# \CU(\BF_p)}$, $$\prod_{p\geqslant M}(1+c_p)=\exp \left(\sum_{p\geqslant M}\log(1+c_p)\right)=1+O\left(\sum_{p\geqslant M}\frac{1}{p^{\codim_X(Z)}}\right)=1+O\left(\frac{1}{M^{\codim_{X}(Z)-1}}\right).$$
We also deduce from \eqref{eq:IU} and Proposition \ref{prop:restrictionTamagawa} that $I_{\CU}= O(\m_{\infty}^X(B_\infty(T)))$, where the implied constant depends only on $B_0,\CU$. 
Gathering together the computations above, we obtain that, for any $M>\max(M_0,M_1)$,
\begin{equation}\label{eq:IXmodulIU}
\begin{split}
I_{\CX}(\CU(\BZ/\mathfrak{P}_{M,l_0}))&=I_{\CU}\prod_{p\geqslant M}(1+c_p)=I_{\CU}+O\left(\frac{I_{\CU}}{M^{\codim_{X}(Z)-1}}\right)=I_{\CU}+O\left(\frac{\m_{\infty}^X(B_\infty(T))}{M^{\codim_{X}(Z)-1}}\right).
\end{split}
\end{equation}
The implied constants above depend only on $B_{\infty},B_0,\CU,\CX$.

Finally, we apply Proposition \ref{prop:HLintegralforsubset} with $$l:=\mathfrak{P}_{M,l_0},\quad  S:=\CU(\BZ/\mathfrak{P}_{M,l_0}).$$  The Lang-Weil estimate \eqref{eq:Lang-WeilZmodl} for $\CU$ implies $$\# S= \#\CU(\BZ/\mathfrak{P}_{M,l_0})=O_\varepsilon(\mathfrak{P}_{M,l_0}^{\dim(X)+\varepsilon}).$$ Therefore using \eqref{eq:IXmodulIU}, we get, 
\begin{align*}
	N_X(B_\infty\times B_{f,0}^{\CX}(\CU(\BZ/\mathfrak{P}_{M,l_0}),	\mathfrak{P}_{M,l_0});T)&=I_{\CX}(\CU(\BZ/\mathfrak{P}_{M,l_0}))+ O_\varepsilon(\mathfrak{P}_{M,l_0}^{\sigma_\CX+\dim(X)+\varepsilon}\m^X_{\infty}(B_\infty(T))^{1-\beta_\CX})\\ &=I_{\CU}+O_\varepsilon\left(\mathfrak{P}_{M,l_0}^{\sigma_\CX+\dim(X)+\varepsilon}\m^X_{\infty}(B_\infty(T))^{1-\beta_\CX}+\frac{\m_{\infty}^X(B_\infty(T))}{M^{\codim_{X}(Z)-1}}\right),
\end{align*} where the implied constant depends only on $B_{\infty},B_0,\CU,\CX$. This finishes the proof.	
\end{proof}
Our second application is to estimate integral points on $\CX$ lying in $\CZ$ modulo any single sufficiently big prime. This will be used in Section \ref{se:er:primeextpoly} for establishing the geometric sieve.
\begin{corollary}\label{co:uppercodim2}
	Under the hypothesis of Proposition \ref{prop:HLintegralforsubset}, we have, uniformly for any prime $p_0$ sufficiently large,
	$$N_{X}(B_{\infty}\times B_{f,0}^{\CX}(\CZ(\BF_{p_0}),p_0);T)=O\left(\frac{\m_{\infty}^X(B_{\infty}(T))}{p_0^{\codim_{X}(Z)}}+p_0^{\sigma_\CX+\dim Z}\m^X_{\infty}(B_\infty(T))^{1-\beta_\CX}\right),$$ where the implied constant may depend on $B_{\infty},B_0,\CX,\CZ$, but it is independent of $p_0$.
\end{corollary}
\begin{proof}
	We keep using the constant $M_1>2$ (depending only on $B_{\infty},B_0,\CX$) in the proof of Corollary \ref{co:appromainterm} such that, $\delta_X|_{B_{\infty}\times B_{f,0}^\CX}$ is the composition of the locally constant function $\delta_{B_\infty}^{M_1}$ \eqref{eq:deltaBM} with the projection $B_{\infty}\times B_{f,0}^\CX\to B_{0,M_1}^\CX$.
	
	We also recall the constant $M_0$ in the proof of Proposition \ref{prop:restrictionTamagawa} depending only on $\CX,\CU$ such that \eqref{eq:mpXmpU} holds for any $p>M_0$. 
	Consequently, for any such $p$, thanks to \eqref{eq:Psil-1}, we compute
	\begin{align*}
			\m_p^X(\Psi_p^{-1}(\CZ(\BF_p)))&=\m_p^X(\CX(\BZ_p))-\m_p^X(\Psi_p^{-1}(\CU(\BF_p)))\\ &=\m_p^X(\CX(\BZ_p))-\m_p^X(\CU(\BZ_p))\\ &=\frac{\#\CX(\BF_p)-\#\CU(\BF_p)}{p^{\dim X}}=\frac{\#\CZ(\BF_p)}{p^{\dim X}}.
	\end{align*}
 Therefore, by the Lang-Weil estimates \eqref{eq:Lang-WeilFp} \eqref{eq:Lang-WeilFpsmooth}, $$\frac{\m_p^X(\Psi_p^{-1}(\CZ(\BF_p)))}{\m_p^X(\CX(\BZ_p))}=\frac{\#\CZ(\BF_p)}{\#\CX(\BF_p)}=O\left(\frac{1}{p^{\codim_{X}(Z)}}\right),$$ the implied constant depending only on $\CX,\CZ$.
	
	For any prime $p_0\nmid l_0$, consider \begin{align*}
		I_{\CX}(\CZ(\BF_{p_0}))&:=\int_{B_\infty(T)\times B_{f,0}^{\CX}(\CZ(\BF_{p_0}),p_0)}\delta_X \operatorname{d}\m^X=\m_{\infty}^X(B_{\infty}(T))\int_{B_{f,0}^{\CX}(\CZ(\BF_{p_0}),p_0)}\delta_{B_\infty} \operatorname{d}\m_f^X.\end{align*}
	Then if $p_0>\max(M_0,M_1)$, following the proof of Corollary \ref{co:appromainterm} above, we have
	\begin{align*}
		I_{\CX}(\CZ(\BF_{p_0}))
		 &=\m_{\infty}^X(B_{\infty}(T))  \left( \int_{B_{0,M_1}^\CX}\delta_{B_\infty}^{M_1}\operatorname{d}\m^X_{M_1} \right) \left( \prod_{p\geqslant M_1,p\neq p_0} \lambda_p^{-1}\m_p^X(\CX(\BZ_p))\right)\lambda_{p_0}^{-1}\m_{p_0}^X(\Psi_{p_0}^{-1}(\CZ(\BF_{p_0})))\\ &=\m_{\infty}^X(B_{\infty}(T))  \left( \int_{B_{0,M_1}^\CX}\delta_{B_\infty}^{M_1}\operatorname{d}\m^X_{M_1} \right) \left( \prod_{p\geqslant M_1} \lambda_p^{-1}\m_p^X(\CX(\BZ_p))\right)\frac{\lambda_{p_0}^{-1}\m_{p_0}^X(\Psi_{p_0}^{-1}(\CZ(\BF_{p_0})))}{\lambda_{p_0}^{-1}\m_{p_0}^X(\CX(\BZ_{p_0}))}\\ &=\int_{B_\infty(T)\times B_{f,0}^\CX}\delta_X \operatorname{d}\m^X\times\frac{\m_{p_0}^X(\Psi_{p_0}^{-1}(\CZ(\BF_{p_0})))}{\m_{p_0}^X(\CX(\BZ_{p_0}))} \\ &=O\left(\frac{\m_{\infty}^X(B_{\infty}(T))}{p_0^{\codim_{X}(Z)}}\right),
	\end{align*} where the implied constant depends only on $B_{\infty},B_0,\CX,\CZ$.
	
	Finally applying Proposition \ref{prop:HLintegralforsubset} to $l:=p_0$ and using again the Lang-Weil estimate \eqref{eq:Lang-WeilFp} for $\CZ$, we get
	\begin{align*}
		N_{X}(B_{\infty}\times B_{f,0}^{\CX}(\CZ(\BF_{p_0}),p_0);T)&=I_{\CX}(\CZ(\BF_{p_0}))+O(\#\CZ(\BF_{p_0})\cdot p_0^{\sigma_\CX}\m^X_{\infty}(B_\infty(T))^{1-\beta_\CX})\\ &=O\left(\frac{\m_{\infty}^X(B_{\infty}(T))}{p_0^{\codim_{X}(Z)}}+p_0^{\sigma_\CX+\dim Z}\m^X_{\infty}(B_\infty(T))^{1-\beta_\CX}\right).
	\end{align*}
This finishes the proof.
\end{proof}

\section{\APHL\ for affine varieties and homogeneous spaces}\label{se:applicationofequitoaffinevar}
The structure of this section is as follows. In \S\ref{se:APHLsufficient} we prove Theorem \ref{thm:HLtwistintegral} which outlines hypotheses (i) and (ii) as a sufficient condition for \APHL. Then in \S\ref{se:homogeneousspaces} we show -- Theorem \ref{thm:Browning-Gorodnikequidistribution} -- the compatibility of the hypothesis (i) with works on effective equidistribution of integral points on nice affine homogeneous spaces which are proven to be Hardy-Littlewood (Theorems \ref{thm:nicealgrps} \& \ref{thm:symmetricHL}), and we then state (Corollary \ref{co:APHLforhomogeneousspaces}) that to show \APHL\ for such varieties it remains to establish the geometric sieve condition (see also Remark \ref{rmk:conj}). In \ref{se:affinequad} we prove Theorems \ref{thm:ngeq4} and \ref{thm:n=3}.

\subsection{A sufficient condition for \APHL}\label{se:APHLsufficient}

In this subsection, let $X\subset \BA^n_{\BQ}$ be an affine, smooth geometrically integral variety with a normalized Tamagawa measure $\m^X=\m^X_\infty\times \m_f^X$. 
Let $\delta_{X}:X(\RA)\to\BR_{\geqslant 0}$ be a locally constant, not identically zero function.

We recall the morphism $\phi_a$ \eqref{eq:phia} and the variety $X_a$ \eqref{eq:Xa}, together with the measure $\m^{X_a}$ and the locally constant function $\delta_{X_a}$ defined on $X_a(\RA)$. We have seen in $\S$\ref{se:HLforintegralmodel} that if $$\CX=\Spec\left(\BZ[x_1,\cdots,x_n]/(f_1(\xx),\cdots,f_r(\xx))\right)$$ is an affine integral model of $X$ over $\BZ$, then the integral model $\CX_a$ ``twisted by $a\in\BN_{\neq 0}$'' for $X_a\cong_\BQ X$ is \begin{equation}\label{eq:CXaaffine}
	\CX_a=\Spec(\BZ[x_1,\cdots,x_n]/(a^{\deg f_1}f_1(a^{-1}\xx),\cdots,a^{\deg f_r}f_r(a^{-1}\xx)).
\end{equation}
We observe that the condition $\CX(\BZ)\neq\varnothing$ implies that $\CX_a(\BZ)\neq\varnothing$ for any $a\in\BN_{\neq 0}$. Note also that, if $X(\BQ)\neq\varnothing$, then for any sufficiently divisible $a\in\BN_{\neq 0}$ (depending on $\CX$), we have $\CX_a(\BZ)\neq\varnothing$. 
\begin{theorem}\label{thm:HLtwistintegral}
Suppose that $X$ is an affine Hardy-Littlewood variety with density $\delta_{X}$ and $X(\BQ)\neq\varnothing$. Assume that 
there is an affine integral model $\CX\subset \BA^n_{\BZ}$ of $X$ with the following property: for any sufficiently divisible $a\in\BN_{\neq 0}$ (depending on $\CX$), the integral model $\CX_a$ (cf. \eqref{eq:CXaaffine}) satisfies:
	
	(i) \emph{(Effective equidistribution condition)} There exist $\sigma_{\CX_a}\geqslant 0$ and $0<\beta_{\CX_a}<1$ such that, for any connected component $B_{\infty}\subset X(\BR)$ and for any congruence neighbourhood $B_f^{\CX_a}(\xi,l)\subset \CX_a(\widehat{\BZ})$  of $\CX_a$,
	\begin{equation}\label{eq:uniformlyHLintegra2}
		\begin{split}
		&N_{X_a}(\phi_{a,\BR}(B_{\infty})\times B_f^{\CX_a}(\xi,l);T)\\ =& \int_{\phi_{a,\BR}(B_\infty(T))\times B_f^{\CX_a}(\xi,l)}\delta_{X_a} \operatorname{d}\m^{X_a}+O(l^{\sigma_{\CX_a}}\m_{\infty}^X(B_\infty(T))^{1-\beta_{\CX_a}}),		\end{split} \quad T\to\infty,
	\end{equation}
	where the implied constant of \eqref{eq:uniformlyHLintegra2} may depend on $B_{\infty},\CX_a$, but it is independent of $l$ and $\xi$;
	
	(ii) \emph{(Geometric sieve condition)} For any closed subset $Z\subset X$ with $\codim_X(Z)\geqslant 2$ and $\CZ_a:=\overline{Z}\subset \CX_a$,
	there exist continuous functions $f_1,f_2:\BR_{>0}\to\BR_{>0}$ such that $f_i(x)=o(1)$ as $x\to\infty$ for $i=1,2$, such that, for every connected component $B_\infty\subset X(\BR)$,
	for every $M>0$, \begin{equation}\label{eq:geomsieve}
		\begin{split}
			&\#\{\XX\in \CX_a(\BZ):\XX\in \phi_{a,\BR}(B_{\infty}(T)),\exists p\geqslant M, \XX\mod p\in \CZ_a(\BF_p)\}\\ =&O\left((f_1(M)+f_2(T))\m^X_\infty (B_\infty(T))\right),
		\end{split}\quad T\to \infty,
	\end{equation}
	where the implied constant may depend on $B_{\infty},\CX_a$ and $Z$, but it is independent of $M$.
	
Then the variety $X$ satisfies \APHL\ with density $\delta_{X}$.
\end{theorem}
\begin{remarks*}
		\hfill
\begin{enumerate}
	\item If $X$ in the theorem above is strongly Hardy-Littlewood, i.e. $\delta_{X}\equiv 1$, then $X$ satisfies \APSHL.
	\item Our formulation of \eqref{eq:geomsieve} is compatible with currently known works on geometric sieves \cite{Poonen,Bhargava,Browning-HB}.
\end{enumerate}
\end{remarks*}
\begin{proof}[Proof of Theorem \ref{thm:HLtwistintegral}]
	{\bfseries Step I:}  Let us fix an open subset $U\subset X$ with $Z:=X\setminus U$ and $\codim_{X}(Z)\geqslant 2$. 
	The map $$B_\infty\mapsto B_\infty^U:=B_\infty\cap U(\BR)$$ is a bijection between connected components of $X(\BR)$ and $U(\BR)$, because $\codim_X(Z)\geqslant 2$. We henceforth fix a $B_\infty$ from now on, which amounts to fixing a connected component $B_\infty^U$ of $U(\BR)$. 
	Apply Lemma \ref{le:factorisationdelta} (i) to $\delta_X$ and let $\delta_{B_{\infty}}:X(\RA_f)\to \BR_{\geqslant 0}$ be the resulting locally constant function for $B_\infty$. We may assume that $B_\infty\cap X(\BQ)\neq\varnothing$, because otherwise the assumption that $X$ being Hardy-Littlewood implies that $\delta_{B_\infty}$ is identically zero on $X(\RA_f)$, and the formula \eqref{eq:NXT} for $U$ trivially holds for $B_\infty^U\times B_f^U$, where $B_f^U\subset U(\RA_f)$ is any compact open subset.
	
	Let $\CX$ be as in the assumption of the theorem. For any $a\in\BN_{\neq 0}$, we define the integral model $\CU_a:= \CX_a\setminus \overline{Z}\subset \CX_a$ for $U_a$, where $\overline{Z}$ stands for the Zariski closure of $Z$ in $\CX_a$. Let $a_0\in\BN_{\neq 0}$ be such that the hypotheses (i) and (ii) hold for every $\CX_a$ with $a_0\mid a$. 
	Since $X(\BQ)\neq\varnothing$, we can find $a_1\in\BN_{\neq 0}$ (depending on $\CX$) with $a_0\mid a_1$ such that $\CX_{a_1}(\BZ)\neq\varnothing$. In particular $\CX_{a}(\BZ)\neq\varnothing$ for every $a_1\mid a$.
We take $l_0\in \BN_{\neq 0}$ with $a_1\mid l_0$ depending only on $\CX,\CU$ such that $\CX\times_{\BZ}\BZ[1/l_0]$, $\CU\times_{\BZ}\BZ[1/l_0]$ are smooth with geometrically integral fibres over $\BZ[1/l_0]$, and that $\prod_{p\nmid l_0}\CU(\BZ_p)\neq\varnothing$ (by the Lang-Weil estimate \eqref{eq:Lang-WeilFpsmooth}). 
	We want to show that, for any $a$ with $l_0\mid a$, the formula \eqref{eq:sec2formulaintegralHLtwista} holds for the connected component $\phi_{a,\BR}(B_\infty^U)$ of $U_a(\BR)$ and for any congruence neighbourhood $B^{\CU_a}_f(\xi,l)$ \eqref{eq:BfX} of $\CU_a$ inside $\CX_a(\widehat{\BZ})$\footnote{Note that $\overline{\CU_a}=\CX_a=\overline{\CX_a}$ in this affine setting, and $\CX_a(\BZ)\neq\varnothing$ for any such $a$.}. Upon breaking $B^{\CU_a}_f(\xi,l)$ into a finite number of smaller congruence neighbourhoods as in the proof of Proposition \ref{prop:sec2integralHL}, we may assume $a\mid l$ (hence $l_0\mid l$), since $\prod_{p\nmid l_0}\CU_a(\BZ_{p})\neq\varnothing$.
	Granting this, Theorem \ref{thm:HLtwistintegral} then follows from Proposition \ref{prop:sec2integralHL}.
	
{\bfseries Step II:}
	 In the remaining of the proof, we work with a fixed $\CX_a$ with $l_0\mid a$ and a fixed congruence neighbourhood $B^{\CU_a}_f(\xi_1,l_1)\subset \CX_a(\widehat{\BZ})$ of $\CU_a$ with $a\mid l_1$. Recall that $\phi_a$ induces an isomorphism $\CX_a\times_{\BZ}\BZ[1/a]\cong\CX\times_{\BZ}\BZ[1/a]$ and the same holds for $\CU_a$. Therefore, $\CX_a\times_{\BZ}\BZ[1/l_1]$ and $\CU_a\times_{\BZ}\BZ[1/l_1]$ are smooth with geometrically integral fibres. By abuse of notation, we shall drop the subscript $a$ in what follows.
	
If $\delta_X$ is identically zero on $B_{\infty}\times B^\CU_f(\xi_1,l_1)$, since $\delta_X$ is locally constant, there exist compact open subsets $B_i\subset X(\RA_f)$ such that $B^\CU_f(\xi_1,l_1)\subset \cup_iB_i$, and $\delta_{B_\infty}$ is identically zero on all $B_i$.
The variety $X$ being Hardy-Littlewood implies that $N_X(B_{\infty}\times B_i;T)=0$ by \eqref{eq:relHL} and \eqref{eq:deltaBinfty}, which means that $X(\BQ)\cap (B_{\infty}\times B_i)=\varnothing$ by (\ref{eq:NXT}). We henceforth get $U(\BQ)\cap (B_{\infty}^U\times B^\CU_f(\xi_1,l_1))\subset X(\BQ)\cap (\cup_i (B_{\infty}\times B_i))=\varnothing$. In this case, the formula \eqref{eq:sec2formulaintegralHLtwista} trivially holds.

From now on, we may assume $\delta_X$ is not identically zero on $B_{\infty}\times B^\CU_f(\xi_1,l_1)$. 
  Then by \eqref{eq:measurerealplace} and \eqref{eq:deltaBinfty},
 \begin{equation}\label{eq:constantintegral}  
 	\begin{split}
 	\int_{B_\infty^U(T)\times B^\CU_f(\xi_1,l_1)}\delta_X|_U \operatorname{d}\m^U&=\m_{\infty}^X(B_{\infty}^X(T)) \int_{B^\CU_f(\xi_1,l_1)}\delta_{B_{\infty}}|_U \operatorname{d}\m^U_f,\end{split}\end{equation} 
 where $B^U_\infty(T)$ is defined by \eqref{eq:BUT} and the finite part integral $\int_{B^\CU_f(\xi_1,l_1)}\delta_{B_{\infty}}|_U \operatorname{d}\m^U_f$ is non-zero.

We take 
\begin{equation}\label{eq:stepB0}
	B_0:=\prod_{p\mid l_1} B_p^\CU(\xi_1,l_1) \subset \prod_{p\mid l_1}\CX(\BZ_p)\cap \prod_{p\mid l_1}U(\BQ_p),
\end{equation}
so that with the notation \eqref{eq:Bf0X}, \begin{equation}\label{eq:stepB1}
	B^\CU_f(\xi_1,l_1)=B_{f,0}^\CU=B_0\times \prod_{p\nmid l_1}\CU(\BZ_p).
\end{equation} For any $M$ sufficiently large such that $p\geqslant M\Rightarrow p\nmid l_1$ (e.g. $M\geqslant l_1+1$), by (\ref{eq:defiBCUMl0}), we have
\begin{equation}\label{eq:stepl10}
 \begin{split}
		&N_U(B_\infty^U\times B^\CU_f(\xi_1,l_1);T)\\=&\#\left(\CU(\BZ)\cap (B_\infty(T)\times B_{f,0}^\CU)\right)\\ =&N_X(B_\infty\times B_{f,0}^{\CX}(\CU(\BZ/\mathfrak{P}_{M,l_1}),\mathfrak{P}_{M,l_1});T) \\ &\quad+O(\#\{\XX\in \CX(\BZ)\cap B_\infty:\|\XX\|\leqslant T,\exists p\geqslant M,\XX\mod p\in \CZ(\BF_p)\}),
	\end{split} \end{equation}
where the implied constant is absolute.

{\bfseries Step III.} Thanks to the hypothesis (i), the formula \eqref{eq:uniformlyHLintegral} in Proposition \ref{prop:HLintegralforsubset} holds for any $B_f^{\CX}(\xi,l)\subset B_{f,0}^\CX=B_0\times\prod_{p\nmid l_1}\CX(\BZ_p)\subset\CX(\widehat{\BZ})$ and the assumption of Corollary \ref{co:appromainterm} is satisfied for $B_0$ \eqref{eq:stepB0} above. We therefore apply Corollary \ref{co:appromainterm} and obtain
\begin{equation*}
\begin{split}
	&N_X(B_\infty\times B_{f,0}^{\CX}(\CU(\BZ/\mathfrak{P}_{M,l_1}),\mathfrak{P}_{M,l_1});T)\\ =&\int_{B_\infty^U(T)\times B_{f,0}^\CU}\delta_X|_U \operatorname{d}\m^U+O_\varepsilon\left(	\mathfrak{P}_{M,l_1}^{\sigma_\CX+\dim(X)+\varepsilon}\m^X_{\infty}(B_\infty(T))^{1-\beta_\CX}+\frac{\m_{\infty}^X(B_\infty(T))}{M^{\codim_{X}(Z)-1}}\right).
\end{split} \end{equation*}
Furthermore, applying the hypothesis (ii) to the error term in \eqref{eq:stepl10}, we obtain
\begin{equation}\label{eq:stepl11}
	N_U(B_\infty^U\times B^\CU_f(\xi_1,l_1);T)=\int_{B_\infty^U(T)\times B_{f,0}^\CU}\delta_X|_U \operatorname{d}\m^U+\operatorname{Er},
\end{equation}
where
$$\operatorname{Er}=O_\varepsilon\left(	\mathfrak{P}_{M,l_1}^{\sigma_\CX+\dim(X)+\varepsilon}\m^X_{\infty}(B_\infty(T))^{1-\beta_\CX}+\frac{\m_{\infty}^X(B_\infty(T))}{M^{\codim_{X}(Z)-1}}+(f_1(M)+f_2(T))\m^X_\infty (B_\infty(T))\right),$$ and the implied constant depends only on $B_{\infty},\CX,Z$ as well as $B_0$ which is in turn determined by $(\xi_1,l_1)$ and $\CU$.
	
	As we assume that $X$ is Hardy-Littlewood, $\m^X_\infty (B_\infty(T))\to \infty$ as $T\to\infty$ by definition. Fix $\varepsilon=1$ and take $M=\eta \log (\m^X_\infty (B_\infty(T)))$ (\emph{a fortiori} $M\to \infty$), where $\eta>0$ depends only on $B_{\infty},\CX,\CU$, so that $$\mathfrak{P}_{M,l_1}^{\sigma_\CX+\dim(X)+1}\leqslant \exp\left((\sigma_\CX+\dim(X)+1)\sum_{p<M}\log p\right)\leqslant\m^X_\infty (B_\infty(T))^{\frac{\beta_\CX}{2}}.$$ Therefore the hypothesis (ii) implies that
	\begin{align*}
		\operatorname{Er} = & O\left(	\m^X_{\infty}(B_\infty(T))^{1-\frac{\beta_\CX}{2}}+\frac{\m_{\infty}^X(B_\infty(T))}{\log^{\codim_{X}(Z)-1} (\m^X_\infty (B_\infty(T)))}+(f_1(\eta\log (\m^X_\infty (B_\infty(T))))+f_2(T))\m^X_\infty (B_\infty(T))\right)\\=& o(\m_\infty^X(B_\infty(T))).
	\end{align*}

Returning to \eqref{eq:stepl11} and  recalling \eqref{eq:stepB1}, we obtain finally $$N_U(B_\infty^U\times B^\CU_f(\xi_1,l_1);T)=\int_{B_\infty^U(T)\times B^\CU_f(\xi_1,l_1)}\delta_X|_U \operatorname{d}\m^U+o(\m_\infty^X(B_\infty(T))),\quad T\to\infty,$$ where the implied constant depends only on $B_{\infty},\CX,\CU,(\xi_1,l_1)$.
Finally taking (\ref{eq:constantintegral}) into account, we obtain the equivalence \eqref{eq:sec2formulaintegralHLtwista}. The proof of the theorem is thus completed.
\end{proof}

\subsection{Affine homogeneous spaces}\label{se:homogeneousspaces}
In this article we consider homogeneous spaces of the following two types.

\subsection*{Type I: nice algebraic groups}
By convention, we call $G$ a \emph{nice algebraic group} if $G$ is a semisimple simply connected $\BQ$-simple linear algebraic group over $\BQ$ with $G(\BR)$ non-compact. We fix a faithful representation $G\hookrightarrow \SL_{n,\BQ}$. We consider in this way $G$ as a closed subvariety of $\SL_{n,\BQ}\subset \BA^{n^2}_\BQ$, and we use euclidean norms from $\BA^{n^2}_\BQ$. According to \cite[\S2.2]{Weil}, $G$ is endowed with an invariant gauge form, to which we can associate a Tamagawa measure $\m^G=\m^G_\infty\times\m^G_f$ (with convergence factors (1)) as a unimodular Haar measure on $G(\RA)$. The measure $\m^G$ induces a (finite) Haar measure $\m_q^G$ on $G(\RA)/G(\BQ)$. For every lattice $\Gamma$ of $G$, i.e., $\Gamma$ is a discrete cofinite subgroup of $G(\BR)$, 
the real part $\m^G_\infty$ induces a (finite) Haar measure $\m_\infty^\Gamma$ on the locally compact topological group $G(\BR)/\Gamma$. Appealing to works \cite{Maucourant,GW}, we have
\begin{theorem}\label{thm:nicealgrps}
	Nice algebraic groups are strongly Hardy-Littlewood.
\end{theorem}
\begin{proof}
	Fix any such group $G$. We firstly infer from \cite[Theorem 1]{Maucourant} and \cite[Theorem 2.7]{GW} that  $\m_{\infty}^G(\{\xx\in G(\BR):\|\xx\|\leqslant T\})\to\infty$ as $T\to\infty$. \footnote{In fact they give asymptotic formulas.} 
	Secondly, \cite[Theorem 2]{Maucourant} and \cite[Theorem 1.2]{GW} show in particular that for any $x_0\in G(\BQ)$ and any irreducible lattice $\Gamma$,
	\begin{equation}\label{eq:latticeptscounting}
		\#\{\gamma\in \Gamma\cdot x_0:\|\gamma\|\leqslant T\}\sim \frac{1}{\m_\infty^\Gamma(G(\BR)/\Gamma)}\m_{\infty}^G(\{\xx\in G(\BR):\|\xx\|\leqslant T\}),\quad T\to\infty.\footnote{To obtain this, we can invoke for example the limits in \cite[Theorems 1 \& 2]{Maucourant} and eliminate the term $T^d(\log T)^e\int_{\End(V)}f\operatorname{d}\mu_\infty$.}
	\end{equation}
	
	In order to establish the strongly Hardy-Littlewood property, according to Proposition \ref{prop:sec2integralHL}, it suffices to prove the formula \eqref{eq:sec2formulaintegralHL} for $G(\BR)\times B^\CG_f(\xi,l)$, where $B^\CG_f(\xi,l)$ is a fixed congruence neighbourhood of a fixed integral model $\CG$ of $G$ over $\BZ$ with $\CG(\widehat{\BZ})\neq\varnothing$ and $\xi=(\xi_p)\in\prod_{p\mid l}\CG(\BZ_{p})$. Consider the congruence subgroup $$K_f(l):=\prod_{p\mid l}K_p(l)\times \prod_{p\nmid l}\CG(\BZ), \quad\text{where}\quad K_p(l):=\{\eta\in\CG(\BZ_{p}):\eta\equiv \id (\mod p^{v_p(l)})\},$$ so that \begin{equation}\label{eq:BCGKf}
		B^\CG_f(\xi,l)=K_f(l)\cdot (\xi\times\prod_{p\nmid l}\id). 
	\end{equation}
	
	We shall follow the strategy of the proof of \cite[Theorem 4.2]{Borovoi-Rudnick}. 
	Let $K(l):=G(\BR)\times K_f(l)$ and $\Gamma(l):=K(l)\cap G(\BQ)$. Then $\Gamma(l)$ is a lattice of $G(\BR)$.
	Since $G(\BR)$ is non-compact, the strong approximation theorem (cf. e.g. \cite[Theorem 1.1]{Cao-Huang}) implies that $K(l)\cdot G(\BQ)=G(\RA)$. We therefore have (cf. \cite[4.5]{Borovoi-Rudnick})
	\begin{align*}
		\m_q^G(G(\RA)/G(\BQ))=\m_q^G(K(l)\cdot G(\BQ)/G(\BQ)) =\m_{\infty}^{\Gamma(l)}(G(\BR)/\Gamma(l))\m_f^G(K_f(l)).
	\end{align*}
The left-hand-side is the Tamagawa number of $G$, which is proved to be $1$ (cf. e.g. \cite[p. 58]{Borovoi-Rudnick}). We therefore conclude that \begin{equation}\label{eq:volumelattice}
	\m_{\infty}^{\Gamma(l)}(G(\BR)/\Gamma(l))^{-1}=\m_f^G(K_f(l)).
\end{equation}
Moreover, $\CG(\BZ)$ is dense in $\CG(\widehat{\BZ})$. So we can choose $x_0\in \CG(\BZ)$ such that $x_0\equiv\xi_p\mod p^{v_p(l)}$ for every $p\mid l$.
	Since $G$ is $\BQ$-simple, the lattice $\Gamma(l)$ is irreducible.
	We now invoke \eqref{eq:latticeptscounting} for $\Gamma(l)$, and we obtain
	\begin{align*}
		N_G(G(\BR)\times B^\CG_f(\xi,l);T)&=\#\{\gamma\in \Gamma(l)\cdot x_0:\|\gamma\|\leqslant T\}\\ &\sim \m_{\infty}^{\Gamma(l)}(G(\BR)/\Gamma(l))^{-1}\m_{\infty}^G(\{\xx\in G(\BR):\|\xx\|\leqslant T\})\\ &=\m_{\infty}^G(\{\xx\in G(\BR):\|\xx\|\leqslant T\})\m_f^G(K_f(l))\\ &=\m_{\infty}^G(\{\xx\in G(\BR):\|\xx\|\leqslant T\})\m_f^G(B^\CG_f(\xi,l)),
	\end{align*} by using \eqref{eq:BCGKf} and \eqref{eq:volumelattice}.
This achieves our goal.
\end{proof}

\subsection*{Type II: nice symmetric spaces}
Let $G$ be a connected semisimple simply connected linear algebraic group over $\BQ$ such that $G(\BR)$ has no compact factors. We choose an almost faithful $\BQ$-representation $\iota:G\rightarrow \GL(W)$, where $W\cong \BQ^n$ is an $n$-dimensional $\BQ$-vector space, so that $\iota(G)$ acts on $W$. In this article, a symmetric space $X$ is isomorphic to a smooth geometrically integral Zariski closed orbit $\iota(G)\cdot v_0$ of a vector $v_0\in W$, whose stabilizer $H$ is symmetric (i.e., there exists a non-trivial $\BQ$-involution $\sigma\in \Aut(G)$ such that $H$ is the fixed point locus of $\sigma$), connected, and has no non-trivial $\BQ$-characters (hence reductive). In particular $X\cong G/H$ and $X(\BQ)\neq\varnothing$. We consider $X$ as a closed subvariety of $\BA^n_\BQ$ via $X\cong \iota(G)\cdot v_0\subset\BA^n_\BQ$. We recall from $\S$\ref{se:normalisedTamagawameasure} that $X$ is equipped with a normalised Tamagawa measure $\m^X=\m_{\infty}^X\times\m_f^X$. The non-compactness of every connected component $B_\infty\subset X(\BR)$ and the unboundness of $\m_{\infty}^X(B_\infty(T)),T\to\infty$ follow from e.g. \cite[Lemma 2.2 \& (2.4)]{Browning-Gorodnik}.  

 In \cite[\S5]{Borovoi-Rudnick} Borovoi and Rudnick introduced a locally constant function
 \begin{equation}\label{eq:deltaX}
 	\delta_{X}: X(\RA)\to \{0,\#(C(H))\},
 \end{equation}
 where $C(H)$ is the dual group of $\Pic(H)$ \emph{à la Kottwitz} (cf. \cite[\S3.4]{Borovoi-Rudnick}). This function is an indicator of adelic orbits containing rational points. That is, for every adelic orbit $\orbit$ of $X(\RA)$ under $G(\RA)$, (cf. \cite[Theorem 3.6]{Borovoi-Rudnick})
 $$\delta_{X}|_{\orbit}>0 \Leftrightarrow \orbit\cap X(\BQ)\neq\varnothing.$$
 Based on works \cite{Duke-Rudnick-Sarnak,EM,EMS} \emph{et al}, Borovoi and Rudnick show that
 \begin{theorem}[\cite{Borovoi-Rudnick} Theorems 5.3 \& 5.4]\label{thm:symmetricHL}
 	Symmetric spaces are Hardy-Littlewood with density $\delta_{X}$ \eqref{eq:deltaX}. They are strongly Hardy-Littlewood provided $C(H)=0$.
 \end{theorem}
  \begin{remark*}
 	The density function $\delta_X$ \eqref{eq:deltaX} can equivalently be defined in terms of the orthogonal locus of the Brauer group $\operatorname{Br}(X)$ (cf. \cite[Propositions 2.2 and 2.10]{CT-Xu}). 
 	See the work of Wei--Xu \cite{Wei-Xu} for various explicit formulas in this spirit.
 \end{remark*}
 By convention, we say $X$ is a \emph{nice symmetric space} if either 
 \begin{enumerate}
 	\item $G$ is $\BQ$-simple;\\ or
 	\item $G\cong \SL_{2,\BQ}\times\SL_{2,\BQ}$ and $H$ is the diagonal subgroup.\footnote{The case (2) corresponds to affine quadrics in four variables of signature $(2,2)$.}
 \end{enumerate}

\medskip
For homogeneous spaces of type (I) or (II), works \cite{Nevo-Sarnak,GW,Browning-Gorodnik} provide explicit uniform error term estimates for counting integral points with congruence conditions. Based on them, we are now in a position to show:
\begin{theorem}\label{thm:Browning-Gorodnikequidistribution}
Nice algebraic groups and nice homogeneous spaces all satisfy the hypothesis (i) of Theorem \ref{thm:HLtwistintegral} with respect to an arbitrary affine integral model over $\BZ$ (inside of $\BA^{n^2}_\BZ$ or $\BA^n_{\BZ}$ depending on the type).
\end{theorem}\begin{proof}
Let $X$ be such a homogeneous space. Fix $\CX$ any affine integral model of $X$ over $\BZ$ as above and consider for any $a\in\BN_{\neq 0}$ the integral model $\CX_a$ of $X_a$ over $\BZ$ (cf. \eqref{eq:CXaaffine}, which is also an integral model of $X$ since $X_a\cong_\BQ X$). As $X(\BQ)\neq\varnothing$, then $\CX_a(\BZ)\neq\varnothing$ if $a$ is sufficiently divisible (depending on $\CX$). Now fix any such $\CX_a$. 

Now we fix $B_f^{\CX_a}(\xi,l)$ (cf. \eqref{eq:BfX}) a congruence neighbourhood of $\CX_{a}$, where $l\geqslant 2$ is an integer and $\xi\in\prod_{p\mid l}\CX_a(\BZ_p)$. Let $\overline{\xi}=(\overline{\xi_p})_{p\mid l}$ be the image of $\xi$ in $\CX_a(\BZ/l)\cong\prod_{p\mid l}\CX_{a}(\BZ/p^{v_p(l)})$ under the residue map $\Psi_l$ \eqref{eq:defPsi}. Since $\CX_a$ is affine, each factor $B_p^{\CX_a}(\xi,l)$ of $B_f^{\CX_a}(\xi,l)$ can equally be defined by means of the congruence conditions in \eqref{eq:congball}. 

Building on the works of Nevo-Sarnak \cite[Theorem 3.2]{Nevo-Sarnak} and Gorodnik-Nevo \cite[Theorem 6.1]{Gorodnik-Nevo} for Type I varieties (using \eqref{eq:BCGKf} \eqref{eq:volumelattice} as in the proof of Theorem \ref{thm:nicealgrps}), and Browning-Gorodnik \cite[Corollaries 2.5 \& 2.6]{Browning-Gorodnik} for Type II varieties, there exist $\sigma_{\CX_a}>0,0<\beta_{\CX_a}<1$ such that, for any connected component $B_\infty\subset X(\BR)$, we have, uniformly for any such $(\xi,l)$,
\begin{align*}
	N_{X_a}(\phi_{a,\BR}(B_{\infty})\times B_f^{\CX_a}(\xi,l);T)&=\#\{\XX\in \CX_a(\BZ)\cap \phi_{a,\BR}(B_{\infty}):\|\XX\|\leqslant T, \text{ for any } p\mid l,\XX\equiv \overline{\xi_p}\mod p^{v_p(l)}\}\\ &=\#\{\XX\in \CX_a(\BZ)\cap \phi_{a,\BR}(B_{\infty}):\|\XX\|\leqslant T, \XX\equiv \overline{\xi}\mod l\}\\ &=\int_{\phi_{a,\BR}(B_\infty)(T)\times B_f^{\CX_a}(\xi,l)}\delta_{X_a} \operatorname{d}\m^{X_a}+O(l^{\sigma_{\CX_a}}\m_{\infty}^{X_a}(\phi_{a,\BR}(B_{\infty})(T))^{1-\beta_{\CX_a}})\\ &=\int_{\phi_{a,\BR}(B_\infty)(T)\times B_f^{\CX_a}(\xi,l)}\delta_{X_a} \operatorname{d}\m^{X_a}+O(l^{\sigma_{\CX_a}}\m_{\infty}^X(B_\infty(T))^{1-\beta_{\CX_a}}).
\end{align*}
Therefore this confirms the hypothesis (i) of Theorem \ref{thm:HLtwistintegral} for the integral model $\CX$. 
\end{proof}
\begin{corollary}\label{co:APHLforhomogeneousspaces}
 Assume that a homogeneous space $X=G/H$ in Theorem \ref{thm:Browning-Gorodnikequidistribution} satisfies the hypothesis (ii) of Theorem \ref{thm:HLtwistintegral} with respect to a fixed affine integral model (inside of $\BA^{n^2}_\BZ$ or $\BA^n_{\BZ}$ depending on the type). Then $X$ satisfies \APHL\ with density $\delta_{X}$ \eqref{eq:deltaX}. If moreover $H$ is simply connected, then $X$ satisfies \APSHL.
\end{corollary}
\begin{proof}
It follows directly from Theorems \ref{thm:HLtwistintegral}, \ref{thm:nicealgrps}, \ref{thm:symmetricHL} that $X$ is Hardy-Littlewood with density $\delta_{X}$. Now let $\CX$ be an affine integral model of $X$, and let $a_1\in\BN_{\neq 0}$ (depending on $\CX$) be such that every integral model $\CX_{a}$ with $a_1\mid a$ satisfies the hypothesis (ii) of Theorem \ref{thm:HLtwistintegral}. By Theorem \ref{thm:Browning-Gorodnikequidistribution}, there exists $a_2\in\BN_{\neq 0}$ (depending on $\CX$) such that every $\CX_{a}$ with $a_2\mid a$ satisfies the hypothesis (i) of Theorem \ref{thm:HLtwistintegral}. Therefore, $\CX_{a}$ satisfies both hypotheses whenever $a_1a_2\mid a$, and hence $X$ satisfies \APHL\ with density $\delta_{X}$ by Theorem \ref{thm:HLtwistintegral}. Moreover, \cite[Theorem 0.3, Corollary 0.3.3]{Borovoi-Rudnick} shows that $X$ is strongly Hardy-Littlewood, provided that $H$ is simply connected.
\end{proof}
\begin{remark}\label{rmk:conj}
	The works \cite[Theorems 1.1, 1.3]{CLX}, \cite[Theorems 1.4, 1.6, 1.7]{Cao-Huang} confirm that many nice affine homogeneous spaces satisfy \textbf{(APSA)}. In light of Theorem \ref{thm:Browning-Gorodnikequidistribution}, we conjecture that they should also satisfy \textbf{(APHL)}, and by Corollary \ref{co:APHLforhomogeneousspaces}, it remains to prove the geometric sieve condition with respect to a fixed integral model for them.
\end{remark}
\subsection{Affine quadrics}\label{se:affinequad}
Returning to the affine quadric $\BFQ$ \eqref{eq:integralmodelaffinequadric}.
We recall that, on fixing $P\in \BFQ(\BQ)$, $\BFQ$ is a symmetric space under $G=\Sp_q$ with stabilizer $H\cong \Sp_q|_{P^\perp}$. 
The group $\Sp_{q}$ is always $\BQ$-simple if $n\geqslant 5$ or $n=3$, whereas when $n=4$, it can happen that $\Sp_{q}\cong \SL_{2,\BQ}\times \SL_{2,\BQ}$, which is the only exceptional case where $\Sp_{q}$ can possibly be not $\BQ$-simple (\cite[Remark 2.4]{Browning-Gorodnik}). In any case, affine quadrics with $n\geqslant 3$ variables are nice symmetric spaces.
	
	Admitting Theorem \ref{thm:geomsieve}, we can now prove our main theorems stated in the introduction.
\begin{proof}[Proof of Theorems \ref{thm:ngeq4} and \ref{thm:n=3}]
	
	Recall that we assume the form $q(\xx)$ defining the affine quadric $\BFQ$ \eqref{eq:affinequadric} to be $\BR$-isotropic. The real locus $\BFQ(\BR)$ of $\BFQ$ is a hyperboloid. It is connected if $n\geqslant 4$, and it can be one sheeted or two sheeted when $n=3$ (depending on $H$).
	It follows from \eqref{eq:globalgrowthCQ} that for any connected component $B_\infty\subset \BFQ(\BR)$, we have
	$$\m_\infty^{\BFQ}(B_\infty(T))\asymp T^{n-2}.$$
	
On the other hand, we take an integral model $\CQ\subset\BA^n_\BZ$ of the form $q(\xx)=m$ (cf. \eqref{eq:integralmodelaffinequadric}) satisfying the hypotheses of Theorems \ref{thm:ngeq4} and \ref{thm:n=3}. Then for any $a\in\BN_{\neq 0}$, on recalling \eqref{eq:CXaaffine}, the twisted integral model $\CQ_a\subset\BA^n_\BZ$ is defined by the equation $$q(\xx)=a^2m.$$ In particular $-a^2m\det q\neq\square$ and hence the hypothesis of Theorem \ref{thm:n=3} is always satisfied for any $a\in\BN_{\neq 0}$. Then Theorem \ref{thm:geomsieve} implies that the hypothesis (ii) of Theorem \ref{thm:HLtwistintegral} holds for all $\CQ_a$ with $$f_1(x)=\frac{1}{x},\quad f_2(x)=\frac{1}{\sqrt{\log x}}.$$

 Therefore the statements of Theorems \ref{thm:ngeq4} and \ref{thm:n=3} follow from Corollary \ref{co:APHLforhomogeneousspaces}.
\end{proof}

\section{The geometric sieve for affine quadrics}\label{se:errorterms}

In this whole section, we are devoted to proving Theorem \ref{thm:geomsieve}. For this purpose, let us fix for the rest of this section the integral model $\CQ$ \eqref{eq:integralmodelaffinequadric} over $\BZ$ for $\BFQ$, and a closed subset $Z\subset\BFQ$ of codimension at least two. Let $\CZ:=\overline{Z}\subset \CQ$.
For any $0<N_1<N_2\leqslant \infty$ sufficiently large, let us define
\begin{equation}\label{eq:VN}
V(T;N_1,N_2):=\#\{\XX\in \CQ(\BZ):\|\XX\|\leqslant T,\exists p\in [N_1,N_2],\XX\mod p\in \CZ(\BF_p)\}.
\end{equation}
With this notation, for any $M>0$, we have $$\#\{\XX\in \CQ(\BZ):\|\XX\|\leqslant T,\exists p\geqslant M,\XX\mod p\in \CZ(\BF_p)\}=V(T;M,\infty).$$
According to the range of $N_1,N_2$, we separate our discussion into three parts, in which \eqref{eq:VN} is treated using different methods. In \S\ref{se:er:primeextpoly} we derive an upper bound for \eqref{eq:VN} valid for arbitrary $N_1,N_2$, which is satisfactory if we take $N_1$ growing to infinity as $T\to\infty$ and $N_2=T^\alpha$ with $\alpha>0$ sufficiently small.
In \S\ref{se:er:primeinter}, on taking $N_1=T^\alpha,N_2=T$, we deal with the residues coming from intermediate primes, by employing various Serre-type uniform bounds for integral points on quadrics. In \S\ref{se:er:geometricsieve} we match together a generalised geometric sieve and our half-dimensional sieve for affine quadrics so as to derive a satisfactory upper bound for the case $N_1=T,N_2=\infty$.
In $\S$\ref{se:proofofgeomsieve} we assemble these bounds together and prove Theorem \ref{thm:geomsieve}.

We shall always throughout assume that $-m\det q\neq\square$  if $n=3$. 
And we shall emphasise uniformity of dependence in the assumption of all statements. Unless expressly stated to the contrary, all implied constants are only allowed to depend on $\mathcal{Q}$ and $Z$.

\subsection{Extension to prime moduli of polynomial growth}\label{se:er:primeextpoly}
The purpose of section is to prove:
\begin{theorem}\label{thm:primepoly}
	For any $1\ll N_1<N_2<\infty$, we have
	\begin{align*}
		V(T;N_1,N_2)=O\left(\frac{T^{n-2}}{N_1^{\codim_{\BFQ}(Z)-1}\log N_1}+T^{(n-2)(1-\beta_{\BFQ})}N_2^{\frac{(3n-2)(n-1)}{2}+\dim Z+1}\right).
	\end{align*}
\end{theorem}
\begin{proof}
	Recall from \S\ref{se:affinequad} that we have $\BFQ\cong G/H$, with $$\dim G=\frac{n(n-1)}{2},\quad \dim H=\frac{(n-1)(n-2)}{2}.$$
	
	Let us fix $l_0\geqslant 2$ such that $\CQ$ is smooth over $\BZ[1/l_0]$.
	Thanks to Theorem \ref{thm:Browning-Gorodnikequidistribution}, the hypothesis of Proposition \ref{prop:HLintegralforsubset} is fulfilled  for $B_0:=\prod_{p\mid l_0}\CQ(\BZ_p)$ with exponent (cf. \cite[Corollary 2.6]{Browning-Gorodnik}) $$\sigma_{\CQ}=\dim H+2\dim G=\frac{(3n-2)(n-1)}{2}.$$ Applying Corollary \ref{co:uppercodim2}, we conclude that for any connected component $B_\infty\subset\BFQ(\BR)$, any $p_0$ sufficiently large with $p_0\nmid l_0$,
	\begin{align*}
			N_{\BFQ}(B_{\infty}\times B_{f,0}^{\CQ}(\CZ(\BF_{p_0}),p_0);T)&\ll \frac{\m_{\infty}^{\BFQ}(B_{\infty}(T))}{p_0^{\codim_{\BFQ}(Z)}}+\m_{\infty}^{\BFQ}(B_{\infty}(T))^{1-\beta_\CQ}p_0^{\sigma_\CX+\dim Z}\\ &=O\left(\frac{T^{n-2}}{p_0^{\codim_{\BFQ}(Z)}}+T^{(n-2)(1-\beta_\CQ)}p_0^{\frac{(3n-2)(n-1)}{2}+\dim Z}\right).
	\end{align*}

	For any $p$, let us now consider 
	\begin{equation}\label{eq:NpT}
	V_p(T;\CZ):=\#\{\XX\in \mathcal{Q}(\BZ):\|\XX\|\leqslant T,\XX\mod p\in \CZ(\BF_p)\}.
	\end{equation}
	Then for $N_1$ sufficiently large, summing $V_{p_0}(T;\CZ)$ over all such $p_0$'s lying in $[N_1,N_2]$, we get the following upper bound for $V(T;N_1,N_2)$:
	\begin{align*}
			V(T;N_1,N_2)&\leqslant \sum_{N_1\leqslant p_0\leqslant N_2}V_{p_0}(T;Z)\\ &=\sum_{N_1\leqslant p_0\leqslant N_2}\sum_{B_\infty\subset \BFQ(\BR)}N_{\BFQ}(B_{\infty}\times B_{f,0}^{\CQ}(\CZ(\BF_{p_0}),p_0);T)\\
			&=\sum_{N_1\leqslant p_0\leqslant N_2}O\left(\frac{T^{n-2}}{p_0^{\codim_{\BFQ}(Z)}}+T^{(n-2)(1-\beta_{\BFQ})}p_0^{\frac{(3n-2)(n-1)}{2}+\dim Z}\right)\\ 
			&\ll\frac{T^{n-2}}{N_1^{\codim_{\BFQ}(Z)-1}\log N_1}+T^{(n-2)(1-\beta_{\BFQ})}N_2^{\frac{(3n-2)(n-1)}{2}+\dim Z+1},
	\end{align*} as desired.
\end{proof}

\subsection{Intermediate primes of polynomial range}\label{se:er:primeinter}
 The goal of this section is to show:
\begin{theorem}\label{th:intermediateprimes}
	 Assume moreover, if $n=3$, that the quadratic form $q$ is anisotropic over $\BQ$ and $-m\det q\neq\square$. Then for any $\alpha>0$,
	\begin{align*}
		V(T;T^\alpha, T)=O_{\alpha}\left( \frac{T^{n-2}}{\log T}\right).
	\end{align*}
\end{theorem}

Our method of obtaining the estimate in Theorem \ref{th:intermediateprimes} is to break the sum into residue classes. For any $l\geqslant 2,\overline{\xi}\in (\BZ/l)^n$, let us consider
\begin{equation}\label{eq:NlT}
V_l(T;\overline{\xi}):=\#\{\XX\in\CQ(\BZ):\|\XX\|\leqslant T,\XX\equiv\overline{\xi}\mod l\}.
\end{equation}
We shall separate our discussion into the cases $n=3$ and $n\geqslant 4$. 
It turns out that we need satisfactory bounds for the quantity \eqref{eq:NlT} with $l=p$, which are uniform with respect to $p$ and $\mathbf{\xi}\in \BF_p^n$, and we need to prove that the contribution from summing over all $p$ in this range is still satisfactory, compared to the order of magnitude $T^{n-2}$. Our argument is inspired by \cite[\S5]{Browning-Gorodnik}. 

\subsubsection{Ingredients}
We record here a uniform estimate for the growth of integral points on quadratic hypersurfaces due to Browning-Gorodnik. This result is also useful in \S\ref{se:polyresp}. For $g(\xx)\in \BQ[x_1,\cdots,x_L]$ a polynomial of degree two, the \emph{quadratic part} $g_0$ of $g$ is the homogeneous degree two part of $g$. Let $\rank(g_0)$ denote the rank of the quadratic form $g_0$. 
\begin{theorem}[\cite{Browning-Gorodnik} Theorem 1.11]\label{thm:BrowningGorodnikaffinequadric}
	For any $\varepsilon>0$, we have, uniformly for any irreducible polynomial $g(\xx)\in\BZ[x_1,\cdots,x_L]$ of degree two with $\rank(g_0)\geqslant 2$, 
	$$\#\{\XX\in\BZ^L:\|\XX\|\leqslant T,g(\XX)=0\}=O_{\varepsilon} (T^{L-2+\varepsilon}).$$
	The implied constant is independent of the polynomial $g$.
\end{theorem}
The case of $L=3$ is a direct consequence of dimension growth bounds obtained by Browning, Heath-Brown and Salberger (cf. \cite[Lemma 4.1]{Browning-Gorodnik}).
\subsubsection{Affine quadrics with three variables}
We start by the more involved case $n=3$.
We first show the following basic fact. Recall that we always assume $\BFQ(\BQ)\neq \varnothing$.
\begin{lemma}\label{le:notsquare}
	The condition $-m\det q\neq\square$ is equivalent to that $\BFQ$ does not contain any line over $\BQ$.
\end{lemma}
\begin{proof}
	 We first note that $\Pic(\BFQ_{\overline{\BQ}})=\BZ$. To see how $\Pic(\BFQ_{\overline{\BQ}})$ is generated, we recall that a projective quadric surface $S$ in $\BP^3$ over $\overline{\BQ}$ has Picard group $\Pic(S)\cong \BZ^2$. Any hyperplane section $E\in|\mathcal{O}(1)|$ intersects $S$ at a conic curve $C$ of divisor type $(1,1)$. So that $\Pic(S\setminus E)\cong \BZ^2/\BZ(1,1)\cong \BZ$. In particular, if the conic curve $C$ splits into two lines, then each of them generates $\Pic(S\setminus E)$. See \cite[II. Examples 6.6.1 \& 6.6.2]{Hartshorne}.
	 
	 So if we compactify $\BFQ$ into $\overline{\BFQ}\subset\BP^3$ and view $\BFQ$ as the complement of some hyperplane section in $\overline{\BFQ}$, then the class of a $\BQ$-line on $\BFQ$ (if any) generates $\Pic(\BFQ)$, since any $\BQ$-plane through that line intersects $\BFQ$ at another $\BQ$-line.
	 Let $d:=-m\det q$ and fix $P\in \BFQ(\BQ)$. Let $H\cong (x^2-dy^2=1)$ be the stabilizer of $\Sp_q$ acting on $P$ (cf. $\S$\ref{se:affinequad}). The group $G=\Gal(\BQ(\sqrt{d}/\BQ))$ operates on $\hat{H}\cong \Pic(\BFQ_{\overline{\BQ}})$, so that $$\Pic(\BFQ)\cong \hat{H}^G.$$ On the one hand, if $d=\square$, then the tangent plane of $\BFQ$ at $P$ intersects $\BFQ$ at two lines over $\BQ$ (cf. \cite[p. 333]{CT-Xu}). On the other hand, if $d\neq\square$, by \cite[p. 331]{CT-Xu}, we have $\Pic(\BFQ)=0$. So we conclude that $\BFQ$ contains no $\BQ$-lines if and only if $d\neq\square$.
\end{proof}

With this at hand, recalling \eqref{eq:NlT}, we now show:

\begin{proposition}\label{prop:intern3}
	Assume that $n=3$. Then under the assumption of Theorem \ref{th:intermediateprimes}, uniformly for any $1\ll p\leqslant T$ and $\overline{\xi}\in \BF_p^3$, we have
	$$V_p(T;\overline{\xi})\ll_{\varepsilon} \left(\frac{T}{p}\right)^\varepsilon\left(1+\frac{T}{p^\frac{4}{3}}\right).$$
\end{proposition}
\begin{proof}
We fix $1\ll p\leqslant T$ and $\overline{\xi}\in \BF_p^3$ in our following arguments.
	Either $V_p(T;\overline{\xi})=0$, for which the desired estimate is evident, or we can find $\XX_0\in \mathcal{Q}(\BZ)$ such that, $\|\XX_0\|\leqslant T,\XX_0\equiv\overline{\xi}\mod p$.
	On making the change of variables $\XX=\XX_0+p\mathbf{y}$, the new variable $\mathbf{y}\in\BZ^3$ satisfies the following equations:
\begin{equation}\label{eq:changevar1}
	\|\yy\|\leqslant \frac{2T}{p},\quad \mathbf{y}\cdot \nabla q(\XX_0)+pq(\mathbf{y})=0.
\end{equation}
	Since $\XX_0$ is a smooth point of $\mathcal{Q}$, we have $\nabla q(\XX_0)\neq \mathbf{0}$. Moreover $p\nmid \gcd(\XX_0)$, since otherwise $p\mid m$, which cannot happen for $p$ large enough. Since $\XX_0\equiv\overline{\xi}\mod p$, the second equation of \eqref{eq:changevar1} implies that $\yy$ lies in the lattice
	$$\Gamma_{\overline{\xi}}^p:=\{\xx\in\BZ^3: \xx\cdot \nabla q(\overline{\xi})\equiv 0\mod p\}$$
	of determinant $\gg\ll p$, the implied constants depending only on $q$  and $m$. \footnote{in fact $\det(\Gamma_{\overline{\xi}}^p)=p$ for $p$ large enough, cf. \cite[p. 1076-1077]{Browning-Gorodnik}.} Thus
\begin{equation}\label{eq:AM}
	V_p(T;\overline{\xi})\leqslant M_{\XX_0,\overline{\xi}}^p(T),
\end{equation}
 where
	$$M_{\XX_0,\overline{\xi}}^p(T):=\#\left\{\yy\in\Gamma_{\overline{\xi}}^p:\|\yy\|\leqslant \frac{2T}{p},\mathbf{y}\cdot \nabla q(\XX_0)+pq(\mathbf{y})=0\right\}.$$
	We are led to bounding  $M_{\XX_0,\overline{\xi}}^p(T)$.

	Choose a minimal basis $\mathbf{L}=(\mathbf{l}_1,\mathbf{l}_2,\mathbf{l}_3)$ of $\Gamma_{\overline{\xi}}^p$ such that (cf. \cite[(5.3)]{Browning-Gorodnik})
\begin{equation}\label{eq:l1l2l3}
	\|\mathbf{l}_1\|\leqslant \|\mathbf{l}_2\|\leqslant \|\mathbf{l}_3\|, \quad \|\mathbf{l}_1\|\|\mathbf{l}_2\|\|\mathbf{l}_3\|\asymp p,
\end{equation}
	so that, on making the non-singular change of variables $\yy\mapsto \mathbf{L}\zz$, the new variable $\zz=(z_1,z_2,z_3)\in\BZ^3$ satisfies, by \eqref{eq:changevar1} and \cite[p. 1075]{Browning-Gorodnik}, 
\begin{equation}\label{eq:changevar2}
	|z_i|\leqslant \frac{c_iT}{p\|\mathbf{l}_i\|},1\leqslant i\leqslant 3,\quad \widetilde{q}(\zz)+\zz\cdot\mathbf{Y}_0=0,
\end{equation}
	where $c_i>0,1\leqslant i\leqslant 3$ are absolute constants, and $$\widetilde{q}(\zz)=q(\mathbf{L}\zz),\quad \mathbf{Y}_0=p^{-1}\mathbf{L}\nabla q(\XX_0).$$

	Now we slice the second equation of \eqref{eq:changevar2} and get for each fixed integer $z_3$ a resulting polynomial $q_{z_3}\in \BZ[z_1,z_2]$. By \eqref{eq:changevar2}, the total number of $z_3$ is \begin{equation}\label{eq:z3}
	\ll 1+\frac{T}{p\|\mathbf{l_3}\|}.
	\end{equation}
	
	We first claim that for any $z_3=\kappa_1\in \BQ$, $\rank((q_{\kappa_1})_0)=2$.
	Indeed, the quadratic part of $q_{\kappa_1}$ is $$(q_{\kappa_1})_0=\widetilde{q}(z_1,z_2,0).$$ The latter, viewed as a quadratic form in two varieties, has rank $1$  (that is, $\rank((q_{\kappa_1})_0)=1$) if and only if $(z_3=0)$ is the tangent plane at a certain point defined over $\BQ$ of the projective quadric $(q=0)\subset \BP^2$. By the assumption that $q$ is $\BQ$-anisotropic, we conclude that this is impossible, which proves the claim. 
	
	We further claim that for any $z_3=\kappa_2\in\BQ$, the polynomial $q_{\kappa_2}$ is irreducible over $\BQ$. Indeed, if the polynomial $q_{\kappa_2}$ is reducible over $\BQ$ for certain $\kappa_2\in\BQ$, that is, it splits into the product of two polynomials $f_1,f_2$ of degree one, then $(z_3=\kappa_2)\cap (f_i=0)$ defines a $\BQ$-line on $\BFQ$ for $i=1,2$ (since the change of variables at each step above is non-singular). This is absurd by Lemma \ref{le:notsquare} as we always assume that $-m\det q\neq\square$. 

	Consequently, using \eqref{eq:changevar2}, Theorem \ref{thm:BrowningGorodnikaffinequadric} shows that the contribution from integral points on each quadric $(q_{z_3}=0)$ with $z_3\in\BZ$ is (as we assume $p\leqslant T$)
	\begin{align*}
		A_{z_3}^p(T)&:=\#\{(z_1,z_2)\in\BZ^2:|z_i|\ll \frac{T}{p\|\mathbf{l}_i\|},1\leqslant i\leqslant 2,q_{z_3}(z_1,z_2)=0\}\\ &=O_\varepsilon\left(1+\left(\frac{T}{p\|\mathbf{l}_1\|}\right)^\varepsilon\right)=O_\varepsilon\left(\left(\frac{T}{p}\right)^\varepsilon\right),
	\end{align*} where the implied constant is independent of $z_3$.
Therefore, taking \eqref{eq:z3} into account, we obtain an upper bound for $M_{\XX_0,\overline{\xi}}^p(T)$ as follows:
	\begin{align*}
	M_{\XX_0,\overline{\xi}}^p(T)&\leqslant \sum_{|z_3|\ll \frac{T}{p\|\mathbf{l_3}\|}}A_{z_3}^p(T)\\ &\ll_{\varepsilon}\left( 1+\frac{T}{p\|\mathbf{l_3}\|}\right)\times \left(\frac{T}{p}\right)^\varepsilon\ll_{\varepsilon}\left(\frac{T}{p}\right)^\varepsilon\left(1+\frac{T}{p^\frac{4}{3}}\right),
	\end{align*}
	because $\|\mathbf{l_3}\|\gg p^\frac{1}{3}$ by \eqref{eq:l1l2l3}.  This finishes the proof, thanks to \eqref{eq:AM}.\end{proof}

\begin{proof}[Proof of Theorem \ref{th:intermediateprimes} for the case $n=3$]
	 We may assume that $Z\neq\varnothing$, hence $\dim Z=0$. We then have $\#\CZ(\BF_p)\leqslant \deg Z$ for every prime $p$. Employing the Prime Number Theorem with partial summation, we have
\begin{equation}\label{eq:primesumaxler}
	\sum_{p\leqslant X} \frac{1}{p^\sigma}\ll_\sigma 	\frac{X^{1-\sigma}}{\log X},\quad 0<\sigma<1.
\end{equation}
	Using Proposition \ref{prop:intern3}, we sum over all primes in the interval $[T^{\alpha},T]$ and we get, for any $0<\varepsilon<\min(\frac{1}{2},\frac{1}{3}\alpha)$,
	\begin{align*}
		V(T;T^\alpha, T)
		\leqslant &\sum_{T^\alpha\leqslant p\leqslant T}\sum_{\overline{\xi}\in \CZ(\BF_p) }V_p(T;\overline{\xi})\\ \ll_{\varepsilon} &(\deg Z)\sum_{T^\alpha\leqslant p\leqslant T}\left(\frac{T^\varepsilon}{p^\varepsilon}+\frac{T^{1+\varepsilon}}{p^\frac{4}{3}}\right)\\
		\ll_\varepsilon &\frac{T^{\varepsilon}\times T^{1-\varepsilon}}{\log T}+T^{1+\varepsilon-\frac{1}{3}\alpha}=O_\varepsilon\left(\frac{T}{\log T}\right).
	\end{align*}
	We fix $\varepsilon>0$ small enough in terms of $\alpha$ so that the implied constant above depends only on $\alpha$. This proves the desired upper-bound.
\end{proof}
\subsubsection{Affine quadrics with at least $4$ variables}
Now we turn to the case $n\geqslant 4$. The analogous version of Proposition \ref{prop:intern3} as a key input for us is the following estimate obtained by Browning and Gorodnik.
\begin{proposition}[\cite{Browning-Gorodnik} Proposition 5.1]\label{prop:intern4}
	Assume $n\geqslant 4$. Then 
	$$V_p(T;\overline{\xi})\ll_{\varepsilon}\left(\frac{T}{p}\right)^{n-3+\varepsilon}\left(1+\frac{T}{p^{\frac{n}{n-1}}}\right)$$
	holds uniformly for any prime $1\ll p\leqslant T$ and for any $\overline{\xi}\in \BF_p^n$.
\end{proposition}
\begin{proof}[Proof of Theorem \ref{th:intermediateprimes} for the case $n\geqslant 4$]
Recalling the Lang-Weil estimate \eqref{eq:Lang-WeilFp} for $\CZ$, we infer from Proposition \ref{prop:intern3} and \eqref{eq:primesumaxler} that
	\begin{align*}
		V(T;T^\alpha, T)\leqslant &\sum_{T^{\alpha}\leqslant p\leqslant T}\sum_{\overline{\xi}\in \CZ(\BF_p) }V_p(T;\overline{\xi})\\ \ll_{\varepsilon}&\sum_{T^{\alpha}\leqslant p\leqslant T} p^{n-3}\times \left(\frac{T}{p}\right)^{n-3+\varepsilon}\left(1+\frac{T}{p^{\frac{n}{n-1}}}\right)\\ \ll_\varepsilon &\sum_{T^{\alpha}\leqslant p\leqslant T} \left(\frac{T^{n-3+\varepsilon}}{p^\varepsilon}+\frac{T^{n-2+\varepsilon}}{p^{1+\frac{1}{n-1}+\varepsilon}}\right)\\ \ll_\varepsilon & T^{n-3+\varepsilon}\times \frac{T^{1-\varepsilon}}{\log T}+T^{n-2+\varepsilon}\times T^{-\frac{\alpha}{n-1}} \ll_{\varepsilon}  \frac{T^{n-2}}{\log T},
	\end{align*}
	for any $0<\varepsilon<\min(\frac{1}{2},\frac{\alpha}{n-1})$. It remains to fix $\varepsilon>0$ small enough depending only on $\alpha$ to get the desired dependence for the implied constant.
\end{proof}

\subsection{Treatment of very large primes}\label{se:er:geometricsieve}
 The goal of this section is to generalise Ekedahl's geometric sieve \cite{Ekedahl}  to affine quadrics.
Our treatment is inspired by a discussion with Tim Browning, to whom we express our gratitude.
\begin{theorem}\label{th:largeprimes}
	Let $N_1=T,N_2=\infty$ in \eqref{eq:VN}. Then
	\begin{align*}
	V(T;T,\infty)=O\left(\frac{T^{n-2}}{(\log T)^{\frac{1}{2}}}\right).
	\end{align*}

\end{theorem}

\begin{proof}[Proof of Theorem \ref{th:largeprimes}]
Upon a change of variables, we may assume that the $\BZ$-integral model $\mathcal{Q}$ is defined by a diagonal quadratic form:
\begin{equation}\label{eq:vm}
\sum_{i=1}^{n} a_ix_i^2=m.
\end{equation}
This may affect \eqref{eq:VN} via a different choice of equivalent height functions in terms of the form $q$, which is clearly negligible.
We may assume that $a_{n-1}\cdot a_n>0$. By multiplying $-1$ to the equation \eqref{eq:vm} if necessary we can assume that both of them are $>0$. We can furthermore assume that $a_n=1$, and all other $a_i$'s are square-free integers, and we write from now on $a_{n-1}=a>0$.

Next, upon enlarging the closed subset $Z$, we claim that, there exist two polynomials $f,g$ which are non-zero, coprime and satisfy $f \in \BZ[x_1,\cdots,x_{n-1}],g\in \BZ[x_1,\cdots,x_{n-2}]$, such that 
\begin{equation}\label{eq:codim2}
	Z=\BFQ\cap(f=g=0)\subset\BA^n_\BQ.
\end{equation} 
In order to verify the claim \eqref{eq:codim2}, let us consider the maps $\pr_1:\BA^n\to\BA^{n-1},\pr_2:\BA^{n-1}\to\BA^{n-2}$, and $\pr_3=\pr_2\circ\pr_1$, the first (resp. second) being the projection onto the first $(n-1)$ (resp. $(n-2)$) coordinates. Since $\dim(Z)\leqslant n-3$, its image $\pr_3(Z)$ has codimension at least one in $\BA^{n-2}_\BQ$. Therefore we can choose a non-zero $g\in \BZ[x_1,\cdots,x_{n-2}]$ such that \begin{equation}\label{eq:functiong}
	\pr_3(Z)\subseteq Z^\prime:=(g=0)\subset \BA^{n-2}_\BQ.
\end{equation} On the other hand, since the map $\pr_1|_{\BFQ}$ is dominant, affine and finite, the subset $\pr_1(Z)\subset\BA^{n-1}_\BQ$ is closed and has codimension at least two, and $\pr_1(Z)\subset \pr_2^{-1}(Z^\prime)=Z^\prime\times\BA^1_\BQ$, the latter being of codimension one in $\BA^{n-1}_\BQ$. We can choose a non-zero $f \in \BZ[x_1,\cdots,x_{n-1}]$ such that $\pr_1(Z)\subset (f=0)\cap\pr_2^{-1}(Z^\prime)$. This gives $Z=\pr_1^{-1}(\pr_1(Z))\subset \BFQ\cap (f=g=0)\subset\BA^n_\BQ$. Note that this procedure may also be reformulated using the classical elimination theory.

 Our goal in the remaining of the proof is to reduce the problem to the affine space $\BA^{n-2}$ via the fibration $\pr_3$ and adapt Theorem \ref{le:Ekedahl} to get a satisfactory control for $V(T;T,\infty)$.
For every $y\in\pr_3(\CZ)\subset \BA_\BZ^{n-2}$, we denote by $\CZ_y:=\CZ\cap\pr_3^{-1}(y)$ the fibre over $y$ as a scheme over $\Spec(k(y))$. Let us break $\pr_3(\CZ)$ into the following two sets.
$$\CY_1:=\{y\in\pr_3(\CZ):\codim_{\BA_{k(y)}^{2}}(\CZ_y)\leqslant 1\},\quad \CY_2:=\{y\in\pr_3(\CZ):\codim_{\BA_{k(y)}^{2}}(\CZ_y)=2\}.$$ 
By the construction \eqref{eq:functiong} and Chevalley's theorem (cf. \cite[Exercise II 3.22 (d)]{Hartshorne}), $\CY_1,\CY_2\subset (g=0)\subset\BA_\BZ^{n-2}$ are constructible subsets.
Moreover by \cite[Theorem 15.1]{Matsumura} we have for $i=1,2$,
\begin{equation}\label{eq:codimYi}
	\codim_{\BA_\BZ^{n-2}}(\CY_i)\geqslant 3-i.
\end{equation}
We can now bound $V(T;T,\infty)$ from above as follows
\begin{equation}\label{eq:V1V2}
	V(T;T,\infty)\leqslant\sum_{i=1}^{2}V_i(T;T,\infty),
\end{equation} where $$V_1(T;T,\infty):=\#\{\XX\in\CQ(\BZ) :\|\XX\|\leqslant T,\text{either }g(\pr_3(\XX))=0\text{ or there exists }p\geqslant T,\pr_3(\XX)\mod p\in \CY_1(\BF_p)\},$$
$$V_2(T;T,\infty):=\#\{\XX\in\CQ(\BZ) :\|\XX\|\leqslant T, g(\pr_3(\XX))\neq 0\text{ and there exists }p\geqslant T,\pr_3(\XX)\mod p\in \CY_2(\BF_p)\}.$$

\textbf{Case I.} Let $\CB_1$ be the set consisting of $\YY\in\BZ^{n-2}$ satisfying \textbf{at least one of} the following conditions:
\begin{itemize}
	\item $g(\YY)=0$;
	\item there exists $p\geqslant T,\YY\mod p\in \CY_1(\BF_p)$.
\end{itemize}
We then have
\begin{align*}
	V_1(T;T,\infty)&\leqslant \#\{\XX=(X_1,\cdots,X_n)\in \mathcal{Q}(\BZ):\|\XX\|\leqslant T,\pr_3(\XX)\in \CB_1\}\\&=\sum_{\substack{\|\YY\|\leqslant T:\YY\in\CB_1}} H_1(\YY)\leqslant \left(\sum_{\substack{\YY\in\BZ^{n-2}:\|\YY\|\leqslant T\\g(\YY)=0}} 1+ \sum_{\substack{\YY\in \BZ^{n-2}:\|\YY\|\leqslant T\\\exists p\geqslant T, \YY\mod p\in \CY_1(\BF_p)}}1\right)H_1(\YY),
\end{align*}
where for any $\YY=(y_1,\cdots,y_{n-2})\in\BZ^{n-2}$, $$H_1(\YY):=\#\{(u,v)\in \BZ^2:u^2+av^2=m-\sum_{i=1}^{n-2}a_iy_i^2\}.$$
We have clearly $H_1(\YY)=O_\varepsilon(T^\varepsilon)$ uniformly for any $\|\YY\|\leqslant T$. 
To deal with the contribution from the sums in bracket, we recall the following quantitative version of Ekedahl's geometric sieve \cite{Ekedahl}.
\begin{theorem}[\cite{Bhargava} Lemma 3.1 \& Theorem 3.3, \cite{Browning-HB} Lemmas 2.1 \& 2.2]\label{le:Ekedahl}
	Let $L,k\geqslant 1$ and let $\CY\subset \BA^L_\BZ$ be a subscheme of codimension $k$. Then  \begin{enumerate}
		\item $\#\{\XX\in \CY(\BZ):\|\XX\|\leqslant T\}=O_\CY(T^{L-k});$
		\item Uniformly for any $M\gg_\CY 1$, $$\#\{\XX\in\BZ^L:\|\XX\|\leqslant T,\exists p\geqslant M,\XX\mod p\in \CY(\BF_p)\}=O_\CY\left(\frac{T^L\log T}{M^{k-1}\log M}+T^{L-k+1}\log T\right).$$
	\end{enumerate}
\end{theorem}
We apply Theorem \ref{le:Ekedahl} (1)  to the codimension one subscheme $(g=0)\subset\BA^{n-2}_\BZ$ for the first sum, and apply Theorem \ref{le:Ekedahl} (2) applied to the subscheme $\CY_1\subset\BA^{n-2}_\BZ$ of codimension at least two (thanks to \eqref{eq:codimYi}) in the remaining sum. We henceforth obtain
\begin{equation}\label{eq:V1infty}
V_1(T;T,\infty)= O_\varepsilon\left(\left(T^{n-3}+\frac{T^{n-2}\log T}{T\log T}+T^{n-3}\log T\right)\times T^\varepsilon\right)=O_\varepsilon(T^{n-3+\varepsilon}).
\end{equation}

\textbf{Case 2.}
Let us consider from now on the set $\CB_2$ of $\YY\in \BZ^{n-2}$ satisfying \textbf{all of} the conditions below:
\begin{itemize}
	\item $g(\YY)\neq 0$;
	\item for all $p\geqslant T$, the polynomial $f(\YY,x_{n-1})\mod p$  in $x_{n-1}$ is a non-zero.
\end{itemize}  
We now show that overall contribution from Case 2 bounds $V_2(T;T,\infty)$ from above. Now take $\XX\in \CQ(\BZ)$ and write $\YY=\pr_3(\XX)=(y_1,\cdots,y_{n-2})$. For every prime $p$, we write $\YY_p:=\pr_3(\XX)\mod p$. On recalling \eqref{eq:codim2}, the fibre $\CZ_{\YY_p}\subseteq\BA_{\BF_p}^2$ is defined by the equations $$m-\sum_{i=1}^{n-2}a_iy_i^2\equiv ax_{n-1}^2+x_n^2\mod p\quad \text{and}\quad f(\YY_p,x_{n-1})\equiv 0\mod p.$$ The condition $\pr_3(\XX)\mod p\in \CY_2(\BF_p)$ implies that  $\codim_{\BA_{\BF_p}^2}(\CZ_{\YY_p})=2$. In particular, if $p>\max_{1\leqslant i\leqslant n-1}(a_i)$, then the polynomial $f(\YY,x_{n-1}) \mod p$ in $x_{n-1}$ is non-zero.

Consequently, we have for $T\gg 1$, $$V_2(T;T,\infty)\leqslant \sum_{\substack{\YY=(y_1,\cdots,y_{n-2})\in\CB_2:\|\YY\|\leqslant T\\\exists u,v\in\BZ,u^2+av^2=m-\sum_{i=1}^{n-2}a_iy_i^2}} H_2(\YY),$$
where for any $\YY=(y_1,\cdots,y_{n-2})\in\BZ^{n-2}$,
$$H_2(\YY):=\sum_{\substack{p:p\geqslant T\\p\mid g(\YY)}}\sum_{\substack{y\in\BZ:|y|\leqslant T\\ p\mid f(\YY,y)}}\#\{z\in\BZ:z^2=m-ay^2-\sum_{i=1}^{n-2}a_iy_i^2\}.$$
Under the assumption that $\YY\in\CB_2$, we have $g(\YY)\neq 0$, and $g(\YY)\ll T^{\deg g}$, so the number of primes $\geqslant T$ dividing $g(\YY)$ is $\ll \deg g$. Hence
$$\sum_{\substack{p:p\geqslant T\\p\mid g(\YY)\neq 0}}1=O(\deg g).$$
Moreover, uniformly for any $p\geqslant T$,
 $$\sum_{\substack{y\in\BZ:|y|\leqslant T\\ p\mid f(\YY,y)}}1\ll(\deg f) \left(\frac{T}{p}+1\right)=O(\deg f),$$
 because $f(\YY,x_{n-1}) \mod p$ is a non-zero polynomial in $x_{n-1}$ and hence has at most $\deg f$ roots over $\BF_p$. All implied constants above depend only on the polynomials $f,g$, that is, the variety $Z$.
So 
\begin{align*}
	H_2(\YY)\ll  \deg f\times \deg g=O(1).
\end{align*}
Returning to the estimation of $V_2(T;T,\infty)$, the bound for $H_2(\YY)$ results in
$$ V_2(T;T,\infty)\ll\#C(T),$$
 where $$C(T):=\{\YY=(y_1,\cdots,y_{n-2})\in\BZ^{n-2}:\|\YY\|\leqslant T,\exists u,v\in\BZ,u^2+av^2=m-\sum_{i=1}^{n-2}a_iy_i^2\}.$$
 
We are reduced to estimating $\#C(T)$.
For this we appeal to Theorem \ref{thm:halfdimsieve}, whose proof will be given in the next section, by setting $Q_1(\xx)=m-\sum_{i=1}^{n-2}a_ix_i^2$ and $Q_2(y_1,y_2)=y_1^2+ay_2^2$. If $n\geqslant 4$, then the affine quadric $(Q_1(\xx)=0)\subset\BA^{n-2}$ is clearly smooth.
When $n=3$, the condition $-m\det q\neq \square $ is equivalent to the stated one in Theorem \ref{thm:halfdimsieve}. So all assumptions of Theorem \ref{thm:halfdimsieve} are satisfied. We thus obtain
\begin{equation}\label{eq:V2infty}
V_2(T;T,\infty)\ll \#C(T)\ll\frac{T^{n-2}}{(\log T)^{\frac{1}{2}}}.
\end{equation}

Finally, thanks to \eqref{eq:V1V2}, the bounds obtained in \textbf{Case 1}  \eqref{eq:V1infty}  and in \textbf{Case 2} \eqref{eq:V2infty} complete the proof. 
\end{proof}

\subsection{Proof of Theorem \ref{thm:geomsieve}}\label{se:proofofgeomsieve}
Keeping the notation in Theorems \ref{thm:primepoly}, \ref{th:intermediateprimes}, \ref{th:largeprimes}, 
it suffices to choose an appropriate parameter $\alpha$ to make these error terms satisfactory.
We therefore fix $\alpha>0$ that satisfies \begin{equation*}
	0<\alpha<\frac{\beta_{\BFQ}(n-2)}{\frac{(3n-2)(n-1)}{2}+\dim Z+1},
\end{equation*} so that $$\alpha\left(\frac{(3n-2)(n-1)}{2}+\dim Z+1\right)+(n-2)(1-\beta_{\BFQ})<n-2.$$

If $M< T^\alpha$, then  with the choice $N_1=M<N_2=T^\alpha$ in Theorem \ref{thm:primepoly}, we have $$V(T;M,T^\alpha)=O\left(\frac{T^{n-2}}{M^{\codim_{\BFQ}(Z)-1}\log M}+T^{\alpha\left(\frac{(3n-2)(n-1)}{2}+\dim Z+1\right)+(n-2)(1-\beta_{\BFQ})}\right).$$
Therefore, we finally obtain, since $ \codim_{\BFQ}(Z)\geqslant 2$,
\begin{align*}
	V(T;M,\infty)\leqslant
	&V(T;M,T^\alpha)+V(T;T^\alpha,T)+V(T; T,\infty)\\ =&O\left(\frac{T^{n-2}}{M^{\codim_{\BFQ}(Z)-1}\log M}+T^{\alpha\left(\frac{(3n-2)(n-1)}{2}+\dim Z+1\right)+(n-2)(1-\beta_{\BFQ})}+\frac{T^{n-2}}{\log T}+\frac{T^{n-2}}{(\log T)^{\frac{1}{2}}}\right)\\ =&O\left(T^{n-2}\left(\frac{1}{M^{\codim_{\BFQ}(Z)-1}\log M}+\frac{1}{(\log T)^{\frac{1}{2}}}\right)\right).
\end{align*}

If $M\geqslant T^\alpha$, then $$V(T;M,\infty)\leqslant V(T;T^\alpha,T)+V(T; T,\infty)=O\left(\frac{T^{n-2}}{(\log T)^{\frac{1}{2}}}\right).$$
This is also satisfactory because $$\frac{T^{n-2}}{M^{\codim_{\BFQ}(Z)-1}\log M}=o\left(\frac{T^{n-2}}{(\log T)^{\frac{1}{2}}}\right).$$
This finishes the proof of Theorem \ref{thm:geomsieve}.
\qed

\section{The half-dimensional sieve for affine quadrics}\label{se:polyresp}
The goal of this section is devoted to proving Theorem \ref{thm:halfdimsieve}. 
Our strategy is based on an upper-bound version of Brun-type half-dimensional sieve, developed in works \cite{Iwaniec,Friedlander-Iwaniec}. This is recorded in $\S$\ref{se:halfdimsievethm}. Keeping the notation in Theorem \ref{thm:halfdimsieve}, let us fix throughout this section a quadratic polynomial $Q_1(\xx)\in\BZ[x_1,\cdots,x_L]$, and a primitive positive-definite non-degenerate quadratic form $Q_2(\yy)\in\BZ[y_1,y_2]$. Let $D_{Q_2}$ be the discriminant of the form $Q_2$. We have $D_{Q_2}\leqslant -3$. In \S\ref{se:resp} we collect some well-known facts about representation of integers by primitive binary quadratic forms of negative discriminant. We shall give full details of the proof of Theorem \ref{thm:halfdimsieve} for the case $L\geqslant 2$ in \S\ref{se:n4resp}, and we indicate necessary modifications in \S\ref{se:n3resp} for the case $L=1$, which is a classical result going back to Bernays \cite{Bernays} (see also \cite[p. 2 Remarks 2 ]{Friedlander-Iwaniec}).

\subsection{Representation by binary quadratic forms}\label{se:resp}
Let us define two arithmetic functions $\bsquare(\cdot)$ and $\bsquarestar(\cdot)$, characterizing integers represented by the form $Q_2(u,v)$. First define
\begin{equation}\label{eq:legendresym}
\CP_{Q_2}:=\left\{p:\left(\frac{D_{Q_2}}{p}\right)=-1\right\},
\end{equation} 
where $\left(\frac{\cdot}{p}\right)$ is the Legendre symbol of modulus $p$.
For $n\in\BZ$, let
\begin{equation}
	\bsquare(n):=\begin{cases}
	1 &\text{ if } \exists u,v\in\BZ,n=Q_2(u,v);\\
	0 &\text{ otherwise},
	\end{cases}
\end{equation}
and
\begin{equation}
	\bsquarestar(n):=\begin{cases}
	1 &\text{ if } \forall p\mid n,p\not\in\CP_{Q_2};\\
	0 &\text{ if } \exists p\in\CP_{Q_2},p\mid n.
	\end{cases}
\end{equation}
 
Fix $n\in\BN_{\neq 0}$. We recall that (see for example \cite[Lemmas 1 \& 2]{James}), if there exists $p\in \CP_{Q_2}$ such that $n=p^{2k+1} m$, with $k\in\BN,m\in\BN_{\neq 0},\gcd(p,m)=1$, then $\bsquare(n)=0$. So, if $\bsquare(n)=1$, then for any $p\in\CP_{Q_2},p\mid n$,  there exists $k\in\BN$, such that $p^{2k}\|n$. We conclude from this analysis that if $\bsquare(n)=1$ then $\bsquarestar(n/r)=1$ with a certain perfect square $r$ dividing $n$, whose prime divisors (if any) are all in $\CP_{Q_2}$. 
The function $\bsquarestar$ is clearly multiplicative, however it is not the case for $\bsquare$ in general. 


\subsection{The half-dimensional sieve}\label{se:halfdimsievethm}
Let $\CA=(a_i)_{i\in I}$ be a finite sequence of integers indexed by $I$. 
For any $d\in\BN_{\neq 0}$, let $\CA_d$ be the subsequence consisting of elements of $\CA$ divisible by $d$. That is, $\CA_d=(a_i)_{i\in I_d}$ with $I_d=\{i\in I:d\mid a_i\}$.
Let $\CP$ be a subset of prime numbers.
For any $z>1$, define the sifting function
$$S(\CA,\CP,z):=\#\{i\in I:\gcd(a_i,\prod_{\substack{p:p\in\CP\\ p<z}}p)=1\}.$$
We will employ the following version of the half-dimensional sieve due to Friedlander and Iwaniec.
\begin{theorem}[\cite{Iwaniec} Theorem 1, \cite{Friedlander-Iwaniec} Lemma 1]\label{thm:I-F}
	There exists a continuous function $G:\mathopen]0,\infty[\to\BR_{>0}$ such that, for any multiplicative arithmetic function $\rho$ satisfying \begin{itemize}
		\item $0\leqslant\rho(p)< p$ for any $p\in\CP$;
		\item there exists $K>0$ such that
		\begin{equation}\label{eq:half-dim}
		\left|\sum_{\substack{p:p\in\CP\\ p<x}}\frac{\rho(p)}{p-\rho(p)}\log p-\frac{1}{2}\log x\right|\leqslant K
		\end{equation}
		for any $x\geqslant 2$,
	\end{itemize} 
	we have, for any $y,z\geqslant 2$, 
	$$S(\CA,\CP,z)\leqslant  \Lambda\left(G\left(w\right)+\mathcal{D}_{K,w}(\log y)^{-
	\frac{1}{5}}\right)\prod_{\substack{p:p\in\CP\\p<z}}\left(1-\frac{\rho(p)}{p}\right)+\sum_{\substack{d<y\\p\mid d\Rightarrow p<z,p\in\CP}}|\mu(d)R(d)|,$$
where $$w:=\frac{\log y}{\log z},\quad \Lambda:=\#\CA(=\# I),\quad  R(d):=\# \CA_d-\frac{\rho(d)}{d}\Lambda,$$
and $\mathcal{D}_{K,w}>0$ depends only on $K,w$.
\end{theorem}

\subsection{Proof of Theorem \ref{thm:halfdimsieve} for the case $L\geqslant 2$}\label{se:n4resp}
With the notation in $\S$\ref{se:resp}, our ultimate goal is to deduce the desired non-trivial upper bound for 
\begin{equation}\label{eq:goalhalfsieve}
\sum_{\XX\in\BZ^L:\|\XX\|\leqslant T} \bsquare(Q_1(\XX))
\end{equation}
via applying Theorem \ref{thm:I-F} to $S(\CA,\CP,z)$ with appropriately chosen $\rho,\CA,\CP,z$. 
Firstly it is convenient to deal separately with the $\XX$'s with $Q_1(\XX)=0$, because $\mathfrak{b}(0)=1$ but $\mathfrak{b}^*(0)=0$. For this we use the evident upper bound (Theorem \ref{le:Ekedahl} (1))\begin{equation}\label{eq:B1T}
B_1(T):=\#\{\XX\in\BZ^L:\|\XX\|\leqslant T,Q_1(\XX)=0\}\ll T^{L-1}.
\end{equation} 
Now the discussion in \S\ref{se:resp} shows that 
\begin{equation}\label{eq:step1}
\sum_{\XX\in\BZ^L:\|\XX\|\leqslant T} \bsquare(Q_1(\XX))\leqslant B_1(T)+\sum^*_{r} \sum_{\substack{\XX\in\BZ^L:\|\XX\|\leqslant T\\ r\mid Q_1(\XX)}} \bsquarestar(Q_1(\XX)/r),
\end{equation}
where the sum with superscript $*$ means that $r$ is restricted to all non-zero perfect squares whose prime divisors are all in $\CP_{Q_2}$ \eqref{eq:legendresym}.

We consider the sequence
$$\CA(T):=\{Q_1(\XX)\}_{\XX\in\BZ^L:\|\XX\|\leqslant T}.$$ 
For $r\in\BN_{\neq 0}$, we define the subsequence
$$\CA(T)_r:=\{Q_1(\XX)/r\}_{\substack{\XX\in\BZ^L:\|\XX\|\leqslant T\\r\mid Q_1(\XX)}}.$$
Consider the arithmetic multiplicative functions $\omega,\varrho$, defined for $N\in\BN_{\neq 0}$ by
\begin{equation}\label{eq:omegarho}
	\omega(N):=\#\{\xi\in (\BZ/N\BZ)^L:Q_1(\xi)\equiv 0\mod N\},\quad \varrho(N):=\frac{\omega(N)}{N^{L-1}}.
\end{equation}
Then $\omega(N)\ll_{\varepsilon} N^{L-1+\varepsilon}$. 
We have, by the Chinese remainder theorem, that for any $r\in\BN_{\neq 0}$,
\begin{equation}\label{eq:omegaN}
\begin{split}
\#\CA(T)_{r}&=\#\{\XX\in\BZ^L:\|\XX\|\leqslant T,r\mid Q_1(\XX)\}\\&=\sum_{\substack{\xi\in(\BZ/r\BZ)^L\\  Q_1(\xi)\equiv 0\mod r}}\#\{\mathbf{y}\in (r\BZ)^L:\|\mathbf{y}+\xi\|\leqslant T\}\\ &=\omega(r)\left(\frac{2T}{r}+O(1)\right)^L.
\end{split}
\end{equation}
By assumption, the affine variety $(Q_1=0)\subset\BA_\BQ^{L}$. Then there exists $l_0\in\BN_{\neq 0}$ such that the affine scheme $\Spec\left(\BZ[x_1,\cdots,x_L]/(Q_1(\xx))\right)$ is smooth over $\BZ[1/l_0]$ with geometrically integral fibres. So by the Lang-Weil estimate \eqref{eq:Lang-WeilFpsmooth}, for any $p\nmid l_0$, we have
$$\varrho(p)=1+O(p^{-\frac{1}{2}}).$$
Let us define \begin{equation}\label{eq:CP2}
\CP_{Q_2}^\prime=\CP_{Q_2}\setminus \{p:\varrho(p)=p\}.\end{equation}
Since there are at most finitely many primes $p$ satisfying $\varrho(p)=p$,
the primes in the set $\CP_{Q_2}^\prime$ have the same density as those in $\CP_{Q_2}$, namely one half (amongst the prime residues modulo $4D_{Q_2}^\prime$ with $D_{Q_2}^\prime\mid D_{Q_2}$ such that $D_{Q_2}^\prime$ is square-free and $D_{Q_2}/D_{Q_2}^\prime$ is a square (cf. \cite[VI. Propositions 5 \& 14]{Serre})). 
With these notions, for any $r\in\BN_{\neq 0},\lambda>0$, one has
\begin{equation}\label{eq:step1prime}
\sum_{\XX\in\BZ^L:\|\XX\|\leqslant T} \bsquarestar(Q_1(\XX)/r)\leqslant S(\CA(T)_r,\CP_{Q_2},T^\lambda)\leqslant S(\CA(T)_r,\CP_{Q_2}^\prime,T^\lambda).
\end{equation}

We next claim that it suffices to deal with sufficiently small $r$'s in \eqref{eq:step1}, more precisely $r<T^\gamma$ for certain $0<\gamma<\Delta_L:=\frac{1}{4L}$.
Whenever there exists $r=q^2\mid Q_1(\XX)\neq 0$ for some $\|\XX\|\leqslant T$, we always have $r\leqslant |Q_1(\XX)| \ll T^2$, in other words, $q\leqslant\sqrt{|Q_1(\XX)|}\ll T$, and so by \eqref{eq:omegaN},
\begin{align*}
\#\CA(T)_{q^2} &\ll \omega(q^2)\times \left( \left(\frac{T}{q^2}\right)^L+1\right)\\ &\ll_\varepsilon (q^2)^{L-1+\varepsilon}\times\left( \left(\frac{T}{q^2}\right)^L+1\right)\\ &\ll_{\varepsilon}\frac{T^L}{q^{2-\varepsilon}}+q^{2L-2+\varepsilon}.
\end{align*}
Therefore the contribution from all $r\in[T^\gamma,T^{1+2\Delta_L}]$ is 
\begin{equation}\label{eq:bd1}
\sum^*_{r\in[T^\gamma,T^{1+2\Delta_L}]} \sum_{\substack{\XX\in\BZ^L:\|\XX\|\leqslant T\\ r\mid Q_1(\XX)}} \bsquarestar(Q_1(\XX)/r)\leqslant\sum_{q\in[T^\frac{\gamma}{2},T^{\frac{1}{2}+\Delta_L}]}\#\CA(T)_{q^2}\ll_\varepsilon T^{L-\frac{\gamma}{2}+\varepsilon}+T^{L-\Delta_L+\varepsilon}\ll T^{L-\frac{\gamma}{2}+\varepsilon}.
\end{equation}
This is satisfactory compared to the expected leading term $\frac{T^L}{\sqrt{\log T}}$.
Next, for \begin{equation}\label{eq:bdq}
T^{1+2\Delta_L}<r\ll T^2\Leftrightarrow T^{\frac{1}{2}+\Delta_L}<q\ll T,
\end{equation} we regard the $(L+1)$-tuple $(\XX,q)=(X_1,\cdots,X_L,q)\in\BZ^{L+1}$ as an integral point on the affine quadric 
$$Q_{s}(\xx,z):(Q_1(\xx)-sz^2=0)\subset \BA^{L+1},$$
where $s$ is an auxiliary integer parameter satisfying $0\neq s\ll T^{1-2\Delta_L}$, thanks to the preassigned bound \eqref{eq:bdq} for $q$ and the fact that if $Q_1(\XX)=0$ then $\mathfrak{b}^*(Q_1(\XX)/r)=0$ for any $r\in\BN_{\neq 0}$.
We want to insert the uniform upper bound estimate in Theorem \ref{thm:BrowningGorodnikaffinequadric} for the quadrics $Q_{s}(\xx,z)$ with $s\neq 0$. Recall that we assume $L\geqslant 2$, so $\rank(Q_{s})_0\geqslant 2$ whenever $s\neq 0$, and moreover the quadratic polynomial $Q_1(\xx)-sz^2$ is irreducible. Since otherwise $Q_1(\xx)-sz^2=(s_1z+A_1(\xx))(s_2z+A_2(\xx))$ and this would imply that $Q_1(\xx)=A_1(\xx)A_2(\xx)$, a contradiction to the assumption that $(Q_1=0)$ is smooth. So the hypotheses of Theorem \ref{thm:BrowningGorodnikaffinequadric} are satisfied. 
We conclude that the contribution of such $r$'s satisfying \eqref{eq:bdq} is
\begin{equation}\label{eq:bd2}
\begin{split}
\sum^*_{T^{1+2\Delta_L}<r\ll T^2} \sum_{\substack{\XX\in\BZ^L:\|\XX\|\leqslant T\\ r\mid Q_1(\XX)}} \bsquarestar(Q_1(\XX)/r)&\ll \sum_{0\neq  s\ll T^{1-2\Delta_L}}\#\{(\XX,q)\in\BZ^{L+1}:\|(\XX,q)\|\ll T,Q_s(\XX,q)=0\} \\ &\ll_{\varepsilon}\sum_{ s\ll T^{1-2\Delta_L}} T^{L+1-2+\varepsilon}\ll T^{L-2\Delta_L+\varepsilon}.
\end{split}
\end{equation}
This is also satisfactory and proves our claim. Gathering together \eqref{eq:step1} \eqref{eq:step1prime}, equation \eqref{eq:goalhalfsieve} now becomes
\begin{equation}\label{eq:step2}
\begin{split}
\sum_{\XX\in\BZ^L:\|\XX\|\leqslant T} \bsquare(Q_1(\XX))&\leqslant B_1(T)+\sum^*_{r<T^\gamma} \sum_{\substack{\XX\in\BZ^L:\|\XX\|\leqslant T\\ r\mid Q_1(\XX)}} \bsquarestar(Q_1(\XX)/r)+\sum^*_{r\geqslant T^\gamma} \sum_{\substack{\XX\in\BZ^L:\|\XX\|\leqslant T\\ r\mid Q_1(\XX)}} \bsquarestar(Q_1(\XX)/r)\\
&\leqslant B_2(T)+\sum^*_{r<T^\gamma} S(\CA(T)_r,\CP_{Q_2}^\prime,T^\lambda),
\end{split}
\end{equation}
where $$B_2(T)=B_1(T)+\sum^*_{r\geqslant T^\gamma} \sum_{\substack{\XX\in\BZ^L:\|\XX\|\leqslant T\\ r\mid Q_1(\XX)}} \bsquarestar(Q_1(\XX)/r)=O_\varepsilon(T^{L-\frac{\gamma}{2}+\varepsilon})$$ according to \eqref{eq:B1T} \eqref{eq:bd1} \eqref{eq:bd2}, and $0<\gamma<\Delta_L,\lambda>0$ are to be chosen later. Everything now boils down to the estimation of $S(\CA(T)_r,\CP_{Q_2}^\prime,T^\lambda)$. 

Our task is to apply Theorem \ref{thm:I-F} to each $\CA(T)_r$. For this purpose we fix $r\in\BN_{\neq 0}$. If $\omega(r)=0$ then the subsequence $\CA(T)_r$ is empty. We therefore assume from now on that $\omega(r)\neq 0$.
Given that $r<\Delta_L<1$, we can develop  \eqref{eq:omegaN} as
\begin{equation}\label{eq:omegaNr}
\Lambda_r:=\#\CA(T)_r=\frac{\omega(r)}{r^L}(2T)^L+O_\varepsilon(T^{L-1+\varepsilon}).
\end{equation}
We define arithmetic functions
$$\omega_r(N):=\frac{\omega(rN)}{\omega(r)},\quad \varrho_r(N):=\frac{\omega_r(N)}{N^{L-1}}=\frac{\varrho(rN)}{\varrho(r)},\quad N\in\BN_{\neq 0}.$$
Let us now verify that the functions $\omega_r$ and $\varrho_r$ are multiplicative. Indeed, for any $N\in\BN_{\neq 0}$, we factorize $N=N_1N_2,r=r_1r_2$ such that $\gcd(r_1N_1,r_2N_2)=1$. Then 
\begin{align*}
	\omega_r(N)=\frac{\omega(r_1r_2N_1N_2)}{\omega(r_1r_2)}&=\frac{\omega(r_1N_1)}{\omega(r_1)}\frac{\omega(r_2N_2)}{\omega(r_2)}\\ &=\frac{\omega(r_1N_1)\omega(r_2)}{\omega(r_1)\omega(r_2)}\frac{\omega(r_2N_2)\omega(r_1)}{\omega(r_2)\omega(r_1)}\\ &=\frac{\omega(rN_1)}{\omega(r)}\frac{\omega(rN_2)}{\omega(r)}=\omega_r(N_1)\omega_r(N_2).
\end{align*}
By Hensel's lemma, for any $p\nmid l_0$ and $p\mid r$, one has
$\omega(pr)=\omega(r)p^{L-1}$, and by the Chinese remainder theorem, for any $p\nmid r$, $\omega(rp)=\omega(r)\omega(p)$.
Hence
\begin{equation}\label{eq:varrp}
\varrho_r(p)=\begin{cases}
\varrho(p) & \text{ if } p\nmid r;\\
1 &\text{ if } p\mid r.
\end{cases}
\end{equation}
This implies $\varrho_r(p)=1+O(p^{-\frac{1}{2}})$ uniformly for any prime $p$ and any $r\in\BN_{\neq 0}$, so by Mertens' first theorem on arithmetic progressions (cf. e.g. \cite[Theorem 2.2]{Iwaniec-Kolwalski}), the hypotheses of Theorem \ref{thm:I-F}, in particular \eqref{eq:half-dim}, are satisfied for the arithmetic function $\varrho_r$ and the set $\CP_{Q_2}^\prime$ uniformly for any $r$ (that is, the remainder term $K$ in \eqref{eq:half-dim} depends only on $Q_1,Q_2$ and is independent of $r$).
 
We are now in a position to apply Theorem \ref{thm:I-F} to $\varrho_r,\CA(T)_{r},\CP_{Q_2}^\prime$ for each perfect square $r<T^\gamma$ with all prime divisors in $\CP_{Q_2}$, with $z=T^\lambda,y=T^\beta$, for $\lambda,\beta>0$ small enough with $\gamma+\beta<1$. For this purpose  we need to evaluate, for each $d\in\BN_{\neq 0}$, the cardinality of the subsequence $\CA(T)_{rd}$, using \eqref{eq:omegaNr} and the definition of $\omega_r,\varrho_r$:
\begin{align*}
	\#\CA(T)_{rd}&=\frac{\omega(dr)}{(dr)^L}(2T)^L+O_\varepsilon(T^{L-1+\varepsilon})\\ &=\frac{\omega_r(d)}{d^L}\frac{\omega(r)}{r^{L}}(2T)^{L}+O_\varepsilon(T^{L-1+\varepsilon})\\ &=\frac{\varrho_r(d)}{d}\Lambda_r+O_\varepsilon(T^{L-1+\varepsilon}).
\end{align*}
The above computation shows that $$R_r(d):=\#\CA(T)_{rd}-\frac{\varrho_r(d)}{d}\Lambda_r=O_\varepsilon(T^{L-1+\varepsilon}).$$
Therefore, according to Theorem \ref{thm:I-F}, we obtain
\begin{align*}
	&S(\CA(T)_r,\CP_{Q_2}^\prime,T^\lambda)\\ \leqslant & \left(G(\beta/\lambda)+\mathcal{D}_{\lambda,\beta}(\log T)^{-\frac{1}{5}}\right)\Lambda_r\prod_{\substack{p:p\in\CP_{Q_2}^\prime\\p<T^\lambda}}\left(1-\frac{\varrho_r(p)}{p}\right)+\sum_{d<T^{\beta}}|R_r(d)|\\ = & \left(G(\beta/\lambda)+\mathcal{D}_{\lambda,\beta}(\log T)^{-\frac{1}{5}}\right)\left(\frac{\omega(r)}{r^L} \prod_{\substack{p:p\in\CP_{Q_2}^\prime\\p<T^\lambda}}\left(1-\frac{\varrho_r(p)}{p}\right)\right)(2T)^L+O_\varepsilon\left(\sum_{d<T^{\beta}}T^{L-1+\varepsilon}\right),
\end{align*}
where $\mathcal{D}_{\lambda,\beta}>0$ depends only on $\lambda,\beta,Q_1,Q_2$.
Thanks to \eqref{eq:varrp}, the leading term in the last expression, up to the factor $\left(G(\beta/\lambda)+\mathcal{D}_{\lambda,\beta}(\log T)^{-\frac{1}{5}}\right)(2T)^L$, can be written as
$$\left( \prod_{\substack{p:p<T^\lambda\\ p\in \CP_{Q_2}^\prime}}\left(1-\frac{\varrho(p)}{p}\right)\right) \times\left(\frac{\omega(r)}{r^L}\prod_{\substack{p:p\mid r\\p\in \CP_{Q_2}^\prime}}\left(\left(1-\frac{1}{p}\right)\left(1-\frac{\varrho(p)}{p}\right)^{-1}\right)\right).$$

We are finally in a position to evaluate the sum in \eqref{eq:step2} as follows. 
The series $\sum_{r=\square} c_r$, formed by $$c_r:=\frac{\omega(r)}{r^L}\prod_{\substack{p:p\mid r\\p\in \CP_{Q_2}^\prime}}\left(\left(1-\frac{1}{p}\right)\left(1-\frac{\varrho(p)}{p}\right)^{-1}\right)=\frac{\omega(r)}{r^L}\prod_{\substack{p:p\mid r\\p\in \CP_{Q_2}^\prime}}\frac{p-1}{p-\varrho(p)},$$ converges, because
$$c_r\ll_{\varepsilon} \frac{r^{L-1+\varepsilon}}{r^{L}}\times r^\varepsilon\ll_{\varepsilon}\frac{1}{r^{1-\varepsilon}}.$$
Therefore,
\begin{equation}\label{eq:step3}
\begin{split}
&\sum_{\substack{r=\square, r<T^\gamma\\ p\mid r\Rightarrow p\in \mathcal{P}_{Q_2}}}S(\CA(T)_r,\CP_{Q_2}^\prime,T^\lambda)\\
&\ll_\varepsilon \left(\left(G(\beta/\lambda)+\mathcal{D}_{\lambda,\beta}(\log T)^{-\frac{1}{5}}\right)(2T)^L\prod_{\substack{p:p<T^{\lambda}\\ p\in \CP_{Q_2}^\prime}}\left(1-\frac{\varrho(p)}{p}\right)\right)\left(\sum_{\substack{r=\square, r<T^\gamma\\ p\mid r\Rightarrow p\in \mathcal{P}_{Q_2}}}c_r\right) +\sum_{\substack{r,d\in\BN_{\neq 0}\\ r<T^\gamma,d<T^{\beta }}} T^{L-1+\varepsilon}\\ &=O_{\gamma,\lambda,\beta,\varepsilon}\left( \frac{T^L}{\sqrt{(\log T)}}\left(1+(\log T)^{-\frac{1}{5}}\right) +T^{L-1+(\gamma+\beta)+\varepsilon}\right).
\end{split}
\end{equation}
Upon choosing $\lambda,\gamma,\beta,\varepsilon>0$ small enough, this finishes the proof of Theorem \ref{thm:halfdimsieve} for the case $L\geqslant 2$. \qed

\begin{remark}
	It would be interesting to ask whether a lower bound of expected magnitude $\frac{T^L}{\sqrt{\log T}}$ exists for \eqref{eq:goalhalfsieve}, just as was established in \cite[Theorem 1]{Friedlander-Iwaniec} for the case $L=1$ under some mild assumptions.					
\end{remark}
\subsection{Sketch of proof of Theorem \ref{thm:halfdimsieve} for the case $L=1$}\label{se:n3resp}
In this case it is equivalent to showing the following.
	For $b_1,b_2\in\BZ_{\neq 0}$, assume that $-D_{Q_2}b_1b_2\neq\square$. Then \begin{equation}\label{eq:respL=1}
		\#\{x\in \BZ:|x|\leqslant T:\exists u,v\in\BZ,b_1x^2+b_2=Q_2(u,v)\}\ll \frac{T}{\sqrt{\log T}}.
	\end{equation}
	We may assume that $b_1>0$, since the set above has finite cardinality if $b_1<0$ and the upper bound \eqref{eq:respL=1} is trivially satisfied.
 Most reasoning (especially the treatment of error terms) is akin to \S\ref{se:n4resp}. We choose to only outline how the dominant term comes out, which is the major difference between these two cases.  Let $\mathfrak{D}$ be the square-free part of $-b_1b_2$. 
 
Let us first assume $\mathfrak{D}=1$. Upon change of variables we are reduced to the case where $b_1=1$ and $b_2=-b^2=\square$, so that the identity inside \eqref{eq:respL=1} is written as $$(x+b)(x-b)=Q_2(u,v).$$ Since $\gcd(x+b,x-b)\mid 2b$, on defining $$L_1(x):=x+b,\quad L_2(x):=x-b,$$ and $$\CP_{Q_2}^\triangle:=\CP_{Q_2}\cup\{p:p\mid 2b\},\quad 	\mathfrak{b}^{*\triangle}(n)=\begin{cases}
	1 &\text{ if } \forall p\mid n,p\not\in\CP_{Q_2}^\triangle;\\
	0 &\text{ if } \exists p\in\CP_{Q_2}^\triangle,p\mid n,
\end{cases}$$ then our discussion in $\S$\ref{se:resp} shows that \eqref{eq:step1} can be modified correspondingly as
$$\sum_{\X\in\BZ:\|\X\|\leqslant T} \bsquare(Q_1(\X))\leqslant B_1(T)+\sum^*_{r^\triangle} \sum_{\substack{\X\in\BZ:\|\X\|\leqslant T\\ r^\triangle\mid L_1(\X)}} \mathfrak{b}^{*\triangle}(L_1(\X)/r^\triangle)+\sum^*_{r^\triangle} \sum_{\substack{\X\in\BZ:\|\X\|\leqslant T\\ r^\triangle\mid L_2(\X)}} \mathfrak{b}^{*\triangle}(L_2(\X)/r^\triangle),$$ where $B_1(T)$ is the same as \eqref{eq:B1T} and the sums with superscript $*$ are over non-zero squared integers $r^\triangle$ such that $p\mid r^\triangle \Rightarrow p\in \CP_{Q_2}^\triangle$. The remaining argument is identical and thus omitted. 

Let us now discuss the case where $\mathfrak{D}\neq 1$, which means that the polynomial $Q_1(\xx)$ is irreducible over $\BQ$. In particular the Legendre symbol $\left(\frac{\mathfrak{D}}{\cdot}\right)$ is non-constant. 
Let us keep using the notation $\CP_{Q_2}$ \eqref{eq:legendresym} and  $\omega(\cdot),\varrho(\cdot)$ \eqref{eq:omegarho} defined in \S\ref{se:n4resp}. Then for any $p\nmid 2b_1b_2$, 
$$\omega(p)=\varrho(p)=\begin{cases}
	2 &\text{ if } \left(\frac{\mathfrak{D}}{p}\right)=1;\\
	0 &\text{ otherwise}.
\end{cases}$$
Since $-D_{Q_2}b_1b_2\neq\square$, the map \begin{align*}
	\{p \text{ prime}:p\nmid 2b_1b_2D_{Q_2}\}&\longrightarrow \{\pm 1\}^2\\
	p &\longmapsto \left(\left(\frac{D_{Q_2}}{p}\right),\left(\frac{\mathfrak{D}}{p}\right)\right)
\end{align*} is surjective. On the one hand if $|D_{Q_2}\mathfrak{D}|\neq \square$, let $D_0$ be the square-free part of $D_{Q_2}\mathfrak{D}$. Therefore we have $|D_0|\neq 1$, and so primes in exactly one quarter of the residue classes modulo $4|D_0|$ satisfy
\begin{equation}\label{eq:oppositesign}
	\left(\frac{D_{Q_2}}{p}\right)=-\left(\frac{\mathfrak{D}}{p}\right)=-1.
\end{equation}
On the other hand, if $|D_{Q_2}\mathfrak{D}|=\square$, then in this case primes in exactly one quarter of the residue classes modulo $4|\mathfrak{D}|$ satisfy \eqref{eq:oppositesign}. To summarize, recall $\CP_{Q_2}^\prime$ \eqref{eq:CP2} and let $$\CP_{Q_2}^{''}:=\CP_{Q_2}^\prime\setminus\{p:p\mid 2b_1b_2\}.$$ 
Then by Mertens' theorem regarding primes in arithmetic progressions, one deduces that
$$\sum_{\substack{p:p\in\CP_{Q_2}^{''}\\ p<x}}\frac{\varrho(p)}{p-\varrho(p)}\log p\sim  2\sum_{\substack{p:\left(\frac{D_{Q_2}}{p}\right)=-\left(\frac{\mathfrak{D}}{p}\right)=-1\\ p<x}}\frac{\log p}{p}\sim 2\times \frac{1}{4}\log x=\frac{1}{2}\log x.$$
Therefore $\varrho$ satisfies the condition \eqref{eq:half-dim} in the half-dimensional sieve (Theorem \ref{thm:I-F}). So the dominant term in \eqref{eq:step3} takes the desired form:
$$T\prod_{\substack{p\in\CP_{Q_2}^{''}\\p<T^\lambda}}\left(1-\frac{\varrho(p)}{p}\right)\asymp T\prod_{\substack{p:\left(\frac{D_{Q_2}}{p}\right)=-\left(\frac{\mathfrak{D}}{p}\right)=-1\\p<T^\lambda}}\left(1-\frac{2}{p}\right)\asymp \frac{T}{\sqrt{\log T}}.$$

\begin{remark}\label{rmk:isotropicsquare}
	To further clarify Remark \ref{rmk:notsquare} at this point, note that if $-D_{Q_2}b_1b_2=\square$, which is equivalent to $D_{Q_2}\mathfrak{D}=\square$, then
	$$\left(\frac{D_{Q_2}}{p}\right)=-1\Leftrightarrow \left(\frac{\mathfrak{D}}{p}\right)=-1.$$
	The main term in the sifting function now grows like
	$$T\prod_{\substack{p\in\CP_{Q_2}^\prime\\p<T^\lambda}}\left(1-\frac{\varrho(p)}{p}\right)=T\prod_{\substack{p\in\CP_{Q_2}\\p<T^\lambda}}\left(1-\frac{2}{p}\right)\asymp T,$$
	which does not give the log saving.
\end{remark}

\section*{Acknowledgments}
We are grateful to Mikhail Borovoi, Zeév Rudnick and Olivier Wittenberg  for their interest in our work.
We would like to address our gratitude to Ulrich Derenthal for his generous support at Leibniz Universit\"at Hannover.
We are in debt to Tim Browning for an enlightening discussion and to the anonymous referees for critical comments, which lead to overall improvements of various preliminary versions of this paper. Part of this work was carried out and reported during a visit to the University of Science and Technology of China. We thank Yongqi Liang for offering warm hospitality. The first author was supported by a Humboldt-Forschungsstipendium. The second author was supported by grant DE 1646/4-2 of the Deutsche Forschungsgemeinschaft.

\end{document}